\documentclass[11pt,a4paper]{article}

\usepackage{fancyhdr}
\usepackage{amsmath,mathtools}
\usepackage{bm}
\usepackage{esint}
\usepackage{amsfonts}
\usepackage{amsthm}
\usepackage{amssymb}
\usepackage{amsbsy}
\usepackage{verbatim}
\usepackage{graphicx}
\usepackage{url}
\usepackage{mathrsfs}
\usepackage[hmargin = 1in,vmargin=.75in]{geometry}
\usepackage{color}
\usepackage{amstext}
\usepackage{stmaryrd}
\usepackage{enumerate}
\usepackage{algpseudocode}
\usepackage{algorithm}
\usepackage{pifont}
\usepackage{prettyref}
\usepackage{subcaption}

\font\msbm=msbm10

\numberwithin{equation}{section}

\theoremstyle{plain}
\newtheorem{theorem}{Theorem}[section]
\newtheorem{lemma}[theorem]{Lemma}

\newtheorem{condition}[theorem]{Condition}
\newtheorem{proposition}[theorem]{Proposition}

\newtheorem{remark}[theorem]{Remark}
\def\mathbb#1{\hbox{\msbm{#1}}}

\newcommand{\bu}{\boldsymbol{u}}
\newcommand{\bv}{\boldsymbol{v}}

\newcommand{\BA}{\boldsymbol{A}}

\newcommand{\BC}{\boldsymbol{C}}

\newcommand{\BG}{\boldsymbol{G}}

\newcommand{\BJ}{\boldsymbol{J}}

\newcommand{\BO}{\boldsymbol{O}}

\newcommand{\BQ}{\boldsymbol{Q}}
\newcommand{\BR}{\boldsymbol{R}}
\newcommand{\BS}{\boldsymbol{S}}

\newcommand{\BU}{\boldsymbol{U}}
\newcommand{\BV}{\boldsymbol{V}}
\newcommand{\BW}{\boldsymbol{W}}
\newcommand{\BX}{\boldsymbol{X}}
\newcommand{\BY}{\boldsymbol{Y}}
\newcommand{\BZ}{\boldsymbol{Z}}

\newcommand{\BPi}{\boldsymbol{\Pi}}
\newcommand{\BDelta}{\boldsymbol{\Delta}}

\newcommand{\BPhi}{\boldsymbol{\Phi}}
\newcommand{\bvarphi}{\boldsymbol{\varphi}}

\newcommand{\BLambda}{\boldsymbol{\Lambda}}
\newcommand{\BSigma}{\boldsymbol{\Sigma}}

\newcommand{\PP}{\mathcal{P}}

\newcommand{\pa}{\partial}

\newcommand{\I}{\boldsymbol{I}}
\newcommand{\RR}{\mathbb{R}}
\newcommand{\lag}{\langle}
\newcommand{\rag}{\rangle}

\newcommand{\eps}{\epsilon}

\DeclareMathOperator{\Tr}{Tr}

\DeclareMathOperator{\mi}{\mathrm{i}}
\DeclareMathOperator{\E}{\mathbb{E}}

\DeclareMathOperator{\diag}{diag}

\DeclareMathOperator{\Null}{null}

\DeclareMathOperator{\St}{St}
\DeclareMathOperator{\op}{op}

\DeclareMathOperator{\rank}{rank}
\DeclareMathOperator{\blkdiag}{blkdiag}

\long\def\\#1//{}

\definecolor{xl}{RGB}{200,50,120}

\begin{document}
\title{\bf Solving Orthogonal Group Synchronization \\ via Convex and Low-Rank Optimization: \\ Tightness and Landscape Analysis}
\author{Shuyang Ling\thanks{New York University Shanghai (Email: sl3635@nyu.edu)}}

\maketitle

\begin{abstract}
Group synchronization aims to recover the group elements from their noisy pairwise measurements. It has found many applications in community detection, clock synchronization, and joint alignment problem. This paper focuses on the orthogonal group synchronization which is often used in cryo-EM and computer vision.  However, it is generally NP-hard to retrieve the group elements by finding the least squares estimator. In this work, we first study the semidefinite programming (SDP)  relaxation of the orthogonal group synchronization and  its tightness, i.e., the SDP estimator is exactly equal to the least squares estimator. Moreover, we investigate the performance of the Burer-Monteiro factorization in solving the SDP relaxation by analyzing its corresponding optimization landscape. 
We provide deterministic sufficient conditions which guarantee: (i) the tightness of SDP relaxation; (ii) optimization landscape arising from the Burer-Monteiro approach is benign, i.e.,  the global optimum is exactly the least squares estimator and no other spurious local optima exist. Our result provides a solid theoretical justification of why the Burer-Monteiro approach is remarkably efficient and effective in solving the large-scale SDPs arising from orthogonal group synchronization. We perform numerical experiments to complement our theoretical analysis, which gives insights into future research directions.
\end{abstract}

\section{Introduction} 

Group synchronization requires to recover the group elements $\{g_i\}_{i=1}^n$ from their partial pairwise measurements:
\[
g_{ij} = g_i^{-1}g_j + w_{ij}, \quad (i,j)\in {\cal E}
\]
where $g_i$ belongs to a given group ${\cal G}$, $w_{ij}$ is the noise, and ${\cal E}$ is the edge set of an underlying network. Depending on the specific group choices, one has found many interesting problems including $\mathbb{Z}_2$-synchronization~\cite{ABBS14}, angular synchronization~\cite{S11,BBS17}, and permutation group~\cite{PKS13}, special orthogonal group~\cite{SS11, AKKSB12}, orthogonal group~\cite{MMMO17},  cyclic group~\cite{CC18}, and real number addition group~\cite{GK06}.

\begin{figure}[h!]
\centering
\includegraphics[width=80mm]{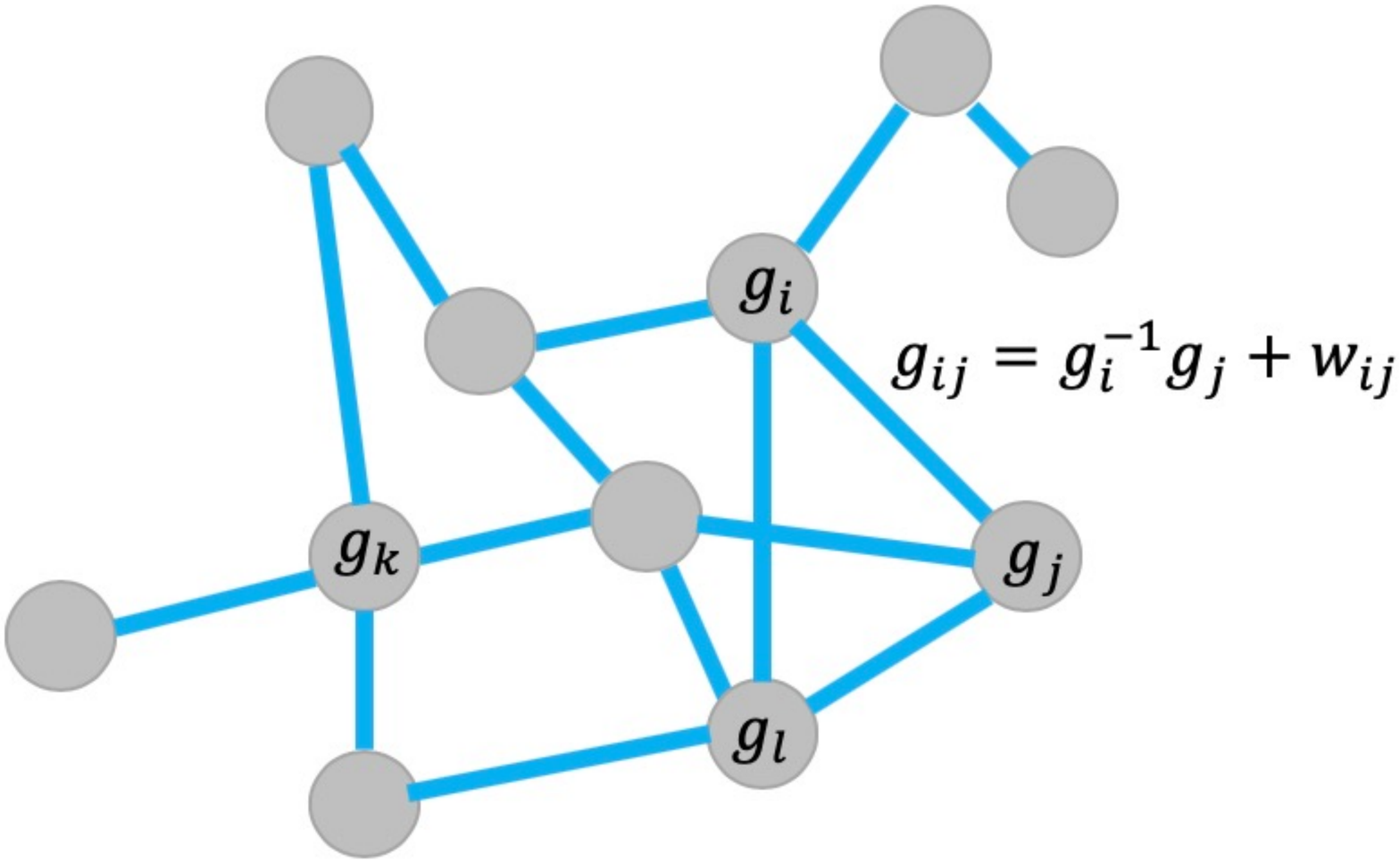}
\caption{Illustration of group synchronization on networks}
\end{figure}

Optimization plays a crucial role in solving group synchronization problems. One commonly used approach is the least squares method. 
However, the least squares objective arising from many of these aforementioned examples are usually highly nonconvex or even inherently discrete. This has posed a major challenge to retrieve the group elements from their highly noisy pairwise measurements because finding the least squares estimator is NP-hard in general. In the recent few years, many efforts are devoted to finding spectral relaxation and convex relaxation (in particular semidefinite relaxation) as well as nonconvex approaches to solve these otherwise NP-hard problems. 
 
In this work, we focus on the general orthogonal synchronization problem:
\begin{equation*}
\BA_{ij} = \BG_i\BG_j^{\top} + \text{noise}
\end{equation*}
where $\BG_i$ is a $d\times d$ orthogonal matrix belonging to
\begin{equation}\label{def:OD}
O(d) : = \{ \BO \in\RR^{d\times d}:  \BO^{\top}\BO =  \BO\BO^{\top}  = \I_d\}.
\end{equation}
Orthogonal group synchronization is frequently found in cryo-EM~\cite{SS11}, computer vision~\cite{AKKSB12} and feature matching problem~\cite{GK06}, and is a natural generalization of $\mathbb{Z}_2$- and angular synchronization~\cite{ABBS14,S11}. We aim to establish a theoretical framework to understand when convex and nonconvex approaches can retrieve the orthogonal group elements from the noisy measurements. In particular, we will answer the following two core questions.
\begin{itemize}
\item For the convex relaxation of orthogonal group synchronization, 
\begin{center}
\emph{When is the convex relaxation tight?}
\end{center}
Convex relaxation has proven to be extremely powerful in approximating or giving the exact solution to many otherwise NP-hard problems under certain conditions. However, it is not always the case that the relaxation yields a solution that matches the least squares estimator. Thus we aim to answer when one can find the least squares estimator of $O(d)$ synchronization with a simple convex program.

\item For the nonconvex Burer-Monteiro approach~\cite{BM03,BM05}, we are interested in answering this question:
\begin{center}
\emph{When does the Burer-Monteiro approach yield a benign optimization landscape?}
\end{center}
Empirical evidence has indicated that the Burer-Monteiro approach works extremely well to solve large-scale SDPs in many applications despite its inherent nonconvexity. One way to provide a theoretical justification is to show that the optimization landscape is~\emph{benign}, i.e., there is only one global minimizer and no other~\emph{spurious} local minima exist.

\end{itemize}

\subsection{Related works and our contribution}
Group synchronization problem is a rich source of many mathematical problems. 
Now we will give a review of the related works which inspire this work.
Table~\ref{tab:sync} provides a non-exhaustive summary of important examples in group synchronization with applications. 

\begin{table}[h!]
\centering
\begin{tabular}{c|c}
\hline
Group & Application \\
\hline
$\mathbb{Z}_2 = \{1,-1\}$ & Community detection and node label recovery~\cite{ABBS14,BVB16} \\ 
\hline
Cyclic group $\mathbb{Z}_n = \mathbb{Z}/n\mathbb{Z}$ & Joint alignment from pairwise differences~\cite{CC18} \\
\hline
$U(1) = \{e^{\mi\theta}: \theta\in [0,2\pi)\}$ & Angular synchronization~\cite{S11,BBS17,ZB18}, Phase retrieval~\cite{IPSV20} \\
\hline
Permutation group & Feature matching~\cite{PKS13,HG13} \\
\hline
SO(3): special orthogonal group & Cryo-EM~\cite{S18,SS11} and computer vision~\cite{AKKSB12} \\
\hline
Addition group on a finite interval & Clock synchronization on networks~\cite{GK06}\\
\hline
\end{tabular}
\caption{Examples of group synchronization and applications}
\label{tab:sync}
\end{table}

The most commonly used approach in group synchronization is to find the least squares estimator. As pointed out earlier, finding the least squares estimator of general $O(d)$ synchronization is NP-hard. In fact, if the group is $\mathbb{Z}_2$, the least squares objective function is closely related to the graph MaxCut problem, which is a well-known NP-hard problem.
Therefore, alternative methods are often needed to tackle this situation. One line of research focuses on the spectral and semidefinite programming relaxation (SDP) including $\mathbb{Z}_2$-synchronization~\cite{ABBS14}, angular synchronization~\cite{S11,BBS17,ZB18}, orthogonal group~\cite{SS11,AKKSB12,CKS15}, permutation group~\cite{PKS13,HG13}. Regarding the SDP relaxation, many efforts are devoted to designing approximation algorithms, such as the Goemans-Williamson relaxation~\cite{GW95} for graph MaxCut.  Inspired by the development of compressive sensing and low-rank recovery~\cite{CRT06, RFP10}, we are more interested in the tightness of convex relaxation: the data are not as adversarial as expected in some seemingly NP-hard problems. Convex relaxation in these problems admits the exact solution to the original NP-hard problem under certain conditions. Following this idea,~\cite{ABBS14} studies the SDP relaxation of $\mathbb{Z}_2$-synchronization with corrupted data; the tightness of SDP for angular synchronization is first investigated in~\cite{BBS17} and a near-optimal performance bound is obtained in~\cite{ZB18}; a very recent work~\cite{Z19} gives a sufficient condition to ensure the tightness of SDP for general orthogonal group synchronization. 
 
Despite the effectiveness of convex approaches, they are not scalable because solving large-scale SDPs are usually very expensive~\cite{N13}. It is more advantageous to keep the low-rank structure and obtain a much more efficient algorithm instead of solving large-scale SDPs directly. As a result, there is a growing interest in developing nonconvex approaches, particularly the first-order gradient-based algorithm. These methods enjoy the advantage of higher efficiency than the convex approach. However, there are also concerns about the possible existence of multiple local optima which prevent the iteration from converging to the global one. The recent few years have witnessed a surge of research in exploring fast and provably convergent nonconvex optimization approaches. Two main strategies are: (i) design a smart initialization scheme and then provide the global convergence; (ii) analyze the nonconvex optimization landscape. Examples include phase retrieval~\cite{CLS15,SQW18}, dictionary learning~\cite{SQW16}, joint alignment~\cite{CC18}, matrix completion~\cite{KOM09,GJZ17} and spiked tensor model~\cite{BMMN19}. 
These ideas are also applied to several group synchronization problems. 
Orthogonal group synchronization can be first formulated as a low-rank optimization program with orthogonality constraints and then tackled with many general solvers~\cite{B15, GLCY18, WY13}. The works~\cite{CC18} on joint alignment and~\cite{ZB18} on angular synchronization follow two-step procedures: first, use the spectral method for a good initialization and then show that the projected power methods have the property of global linear convergence.

Our focus here is on the optimization landscape of the Burer-Monteiro approach~\cite{BM03, BM05} in solving the large SDPs arising from $O(d)$ synchronization. The original remarkable work~\cite{BM05} by Burer and Monteiro shows that as long as $p(p+1)>2n$ where $p$ is the dimension of low-rank matrix and $n$ is the number of constraints, the global optima to Burer-Monteiro factorization match those of the corresponding SDP by using the idea from~\cite{P98}. Later on,~\cite{BVB16, BVB20} show that the optimization landscape is benign, meaning that no spurious local optima exist in the nonconvex objective function if $p$ is approximately greater than $\sqrt{2n}$. This bound is proven to be almost tight in~\cite{WW18}. On the other hand, it is widely believed that even if $p = O(1),$ the Burer-Monteiro factorization works provably, which is supported by many numerical experiments. 
We have benefitted greatly from the works regarding the Burer-Monteiro approach on group synchronization in~\cite{BBV16, X19, MMMO17}. In~\cite{BBV16}, the authors prove that the optimization landscape is benign for $\mathbb{Z}_2$-synchronization as well as community detection under the stochastic block model if $p = 2$. The optimization landscape of angular synchronization is studied in~\cite{X19}. 
The work~\cite{MMMO17} provides a lower bound for the objective function value evaluated at local optima. The bound depends on the rank $p$ and is smartly derived by using the Riemannian Hessian. However, the landscape and the tightness of the Burer-Monteiro approach for $O(d)$ synchronization have not been fully addressed yet, which becomes one main motivation for this work.
It is worth noting that the Burer-Monteiro approach is closely related to the synchronization of oscillators on manifold~\cite{LXB19, MTG17, MTG20}. The analysis of the optimization landscape of the Burer-Monteiro approach in $\mathbb{Z}_2$ with $p=2$ is equivalent to exploring the stable equilibria of the energy landscape associated with the homogeneous Kuramoto oscillators~\cite{LXB19}. This connection is also reflected in the synchronization of coupled oscillators on more general manifolds such as $n$-sphere and Stiefel manifold~\cite{MTG17, MTG20} on arbitrary complex networks.

An important problem regarding the tightness of convex relaxation and landscape analysis is how these two properties depend on the general notion of SNR (signal-to-noise ratio). In most cases, if the noise is rather small compared to the planted signal, optimization methods should easily recover the hidden signal since the tightness of SDP and the benign landscape are guaranteed. However, as the noise strengthens, the landscape becomes bumpy and optimizing the cost function becomes challenging. This leads to the research topic on detecting the critical threshold for this phase transition. Examples can be found in many applications including eigenvectors estimation~\cite{BBP05} and the community detection under the stochastic block model~\cite{A17}. For $\mathbb{Z}_2$- and angular synchronization, convex methods are tight all the way to the information-theoretical limit~\cite{B18, ZB18} but the analysis of optimization landscape remains suboptimal~\cite{BBV16, X19}. Our work on $O(d)$ synchronization will follow a similar idea and attempt to explore the critical threshold to ensure the tightness of convex relaxation and the Burer-Monteiro approach.

Our contribution is multifold. First, we prove the tightness of convex relaxation in solving the $O(d)$ synchronization problem by extending the work~\cite{BBS17, BBV16} on $\mathbb{Z}_2$- and angular synchronization. We propose a deterministic sufficient condition that guarantees the tightness of SDP and easily applies to other noise models. Our result slightly improves the very recent result on the tightness of $O(d)$ synchronization in~\cite{Z19}. Moreover, we analyze the optimization landscape arising from the Burer-Monteiro approach applied to $O(d)$ synchronization. For this low-rank optimization approach, we also provide a general deterministic condition to ensure a benign optimization landscape. 
The sufficient condition is quite general and applicable to several aforementioned examples such as $\mathbb{Z}_2$- and angular synchronization~\cite{BBV16, X19}, and permutation group~\cite{PKS13}, and achieve the state-of-the-art results. Our result on the landscape analysis serves as another example to demonstrate the great success of the Burer-Monteiro approaches in solving large-scale SDPs.

\subsection{Organization of this paper}
The paper proceeds as follows: Section~\ref{s:prelim} introduces the background of orthogonal group synchronization and optimization methods. We show the main results in Section~\ref{s:main}. Section~\ref{s:numerics} focuses on numerical experiments and we give the proofs in Section~\ref{s:proof}.

\subsection{Notation}
For any given matrix $\BX$, $\BX^{\top}$ is the transpose of $\BX$; $\BX\succeq 0$ means $\BX$ is positive semidefinite.
 Denote $\|\BX\|_{\op}$ the operator norm of $\BX$, $\|\BX\|_F$ the Frobenius norm, and $\|\BX\|_*$ the nuclear norm, i.e., the sum of singular values. For two matrices $\BX$ and $\BZ$ of the same size, $\BX\circ \BZ$ denotes the Hadamard product of $\BX$ and $\BZ$, i.e., $(\BX\circ \BZ)_{ij} = X_{ij}Z_{ij}$; $\lag \BX,\BZ\rag : = \Tr(\BX\BZ^{\top})$ is the inner product. ``$\otimes$" stands for the Kronecker product; $\diag(\bv)$ gives a diagonal matrix whose diagonal entries equal $\bv$; $\blkdiag(\BPi_{11},\cdots,\BPi_{nn})$ denotes a block-diagonal matrix whose diagonal blocks are $\BPi_{ii}$, $1\leq i\leq n.$
Let $\I_n$ be the $n\times n$ identity matrix, and $\BJ_n$ be an $n\times n$  matrix whose entries are 1. We denote $\BX\succeq \BZ$ if $\BX - \BZ\succeq 0$, i.e., $\BX-\BZ$ is positive semidefinite. We write $f(n)\lesssim g(n)$ for two positive functions $f(n)$ and $g(n)$ if there exists an absolute positive constant $C$ such that $f(n) \leq C g(n)$ for all $n.$

\section{Preliminaries}\label{s:prelim}

\subsection{The model of group synchronization}
We introduce the model for orthogonal group synchronization. We want to estimate $n$ matrices $\BG_1,\cdots,\BG_n \in O(d)$ from their pairwise measurements $\BA_{ij}$:
\begin{equation}\label{eq:model}
\BA_{ij} = \BG_i\BG_j^{-1} + \BDelta_{ij}, 
\end{equation}
where $\BA_{ij}\in\RR^{d\times d}$ is the observed data and $\BDelta_{ij}\in\RR^{d\times d}$ is the noise. 
Note that $\BG_i^{-1} = \BG_i$ holds for any orthogonal matrix $\BG_i$. Thus in the matrix form, we can reformulate the observed data $\BA$ as
\[
\BA= \BG\BG^{\top} + \BDelta \in\RR^{nd\times nd}
\]
where $\BG^{\top} = [\BG_1^{\top},\cdots,\BG_n^{\top}]\in\RR^{d\times nd}$ and the $(i,j)$-block of $\BA$ is $\BA_{ij} = \BG_i\BG_j^{\top} + \BDelta_{ij}.$ In particular, we set $\BA_{ii} = \I_d$ and $\BDelta_{ii} = 0.$

One benchmark noise model is the group synchronization from measurements corrupted with Gaussian noise, i.e.,
\[
\BA_{ij} = \BG_i\BG_j^{\top} + \sigma\BW_{ij}
\]
where each entry in $\BW_{ij}$ is an i.i.d. standard Gaussian random variable and $\BW_{ij}^{\top} = \BW_{ji}.$
The corresponding matrix form is
\[
\BA = \BG\BG^{\top} + \sigma\BW\in\RR^{nd\times nd}
\]
which is actually a matrix spike model.

One common approach to recover $\BG$ is to minimize the nonlinear least squares objective function over the orthogonal group $O(d)$:
\begin{equation}\label{eq:obj}
\min_{\BR_i\in O(d)} \sum_{i=1}^n\sum_{j=1}^n\| \BR_i \BR_j^{\top} - \BA_{ij}  \|_F^2.
\end{equation}
In fact, the global minimizer equals the maximum likelihood estimator of~\eqref{eq:model} under Gaussian noise, i.e., assuming each $\BDelta_{ij}$ is an independent Gaussian random matrix.

Throughout our discussion, we will deal with a more convenient equivalent form. More precisely, we perform a change of variable:
\begin{align*}
\| \BR_i \BR_j^{\top} - \BA_{ij}  \|_F^2 & = \| \BR_i\BR_j^{\top} - (\BG_i\BG_j^{\top} +  \BDelta_{ij}) \|^2_F \\
& = \| \BG_i^{\top}\BR_i \BR_j^{\top}\BG_j- (\I_d + \BG_i^{\top} \BDelta_{ij} \BG_j )  \|_F^2
\end{align*}
where $\BR_i$ and $\BG_i\in O(d)$.
By letting
\begin{equation}\label{eq:varchange}
\BR_i \leftarrow  \BG_i^{\top}\BR_i, \quad \BDelta_{ij} \leftarrow \BG_i^{\top} \BDelta_{ij} \BG_j
\end{equation}
then the updated objective function becomes
\[
\sum_{i=1}^n\sum_{j=1}^n\| \BR_i\BR_j^{\top} - (\I_d + \BDelta_{ij}) \|_F^2.
\]
Its global minimizer equals the global maximizer to 
\begin{equation}\label{def:OD}
\max_{\BR_i\in O(d)} \sum_{i=1}^n\sum_{j=1}^n \lag\BR_i\BR_j^{\top}, \I_d + \BDelta_{ij} \rag  \tag{P}
\end{equation}
The program (P) is a well-known NP-hard problem.
We will focus on solving (P) by convex relaxation and low-rank optimization approach, and study their theoretical guarantees.

\subsection{Convex relaxation}

The convex relaxation relies on the idea of lifting: let $\BX = \BR\BR^{\top} \in\RR^{nd\times nd}$ with $\BX_{ij} = \BR_i\BR_j^{\top}$. We notice that $\BX\succeq 0$ and $\BX_{ii} = \I_d$ hold for any $\{\BR_i\}_{i=1}^n\in O(d).$ 
The convex relaxation of $O(d)$ synchronization is 
\begin{equation}\label{def:CVX}
\max_{\BX\in\RR^{nd\times nd}}~\lag \BA, \BX\rag \quad \text{such that} \quad \BX_{ii} = \I_d, \quad \BX\succeq 0  \tag{SDP}
\end{equation}
where $(i,j)$-block of $\BA$ is $\BA_{ij} =  \I_d + \BDelta_{ij}\in\RR^{d\times d}.$
In particular, if $d= 1$, this semidefinite programming (SDP) relaxation reduces to the famous Goemans-Williamson relaxation for the graph MaxCut problem~\cite{GW95}.
Since we relax the constraint, it is not necessarily the case that the global maximizer $\widehat{\BX}$ to (SDP) is exactly rank-$d$, i.e., $\widehat{\BX} = \widehat{\BG}\widehat{\BG}^{\top}$ for some $\widehat{\BG}\in\RR^{nd\times d}$ with $\widehat{\BG}_i\in O(d).$ Our goal is to study~\emph{the tightness of this SDP relaxation}: when the solution to~\eqref{def:CVX} is exactly rank-$d$, i.e., the convex relaxation gives the global optimal solution to~\eqref{def:OD} which is also the least squares estimator.

\subsection{Low-rank optimization: Burer-Monteiro approach}
Note that solving the convex relaxation~\eqref{def:CVX} is extremely expensive especially for large $d$ and $n$. Thus an efficient, robust, and provably convergent optimization algorithm is always in great need. 
Since the solution is usually low-rank in many empirical experiments, it is appealing to take advantage of this property: keep the iterates low-rank and perform first-order gradient-based approach to solve this otherwise computationally expensive SDP. In particular, we will resort to the Burer-Monteiro approach~\cite{BM03,BM05} to deal with the orthogonal group synchronization problem. The core idea of the Burer-Monteiro approach is keeping $\BX$ in a factorized form and taking advantage of its low-rank property. Recall the constraints in~\eqref{def:CVX} read $\BX\succeq 0$ and $\BX_{ii} = \I_d$. In the Burer-Monteiro approach, we let $\BX = \BS\BS^{\top}$ where $\BS\in\RR^{nd\times p}$ with $p > d$. We hope to recover the group elements by maximizing $f(\BS)$:
\begin{equation}\label{def:BM}
f(\BS) : = \lag \BA, \BS\BS^{\top}\rag =  \sum_{i=1}^n\sum_{j=1}^n \lag \BA_{ij}, \BS_i\BS_j^{\top} \rag, \tag{BM}
\end{equation}
where 
\[
\BS^{\top} := [\BS_1^{\top}, \cdots, 
\BS_n^{\top} ]\in\RR^{p\times nd}, \quad \BS_i\BS_i^{\top} = \I_d.
\]
In other words, we substitute $\BR_i$ in~\eqref{def:OD} by a partial orthogonal matrix $\BS_i\in\RR^{d\times p}$ with $\BS_i\BS_i^{\top} = \I_d$ and $ p > d$. Therefore,
$\BS$ belongs to the product space of Stiefel manifold, i.e., $\St(d,p)^{\otimes n}:=\underbrace{\St(d,p)\times \cdots \times \St(d,p)}_{n \text{ times}}$
\begin{equation}\label{def:St}
\St(d,p) : = \{\BS_i\in\RR^{d\times p}: \BS_i\BS_i^{\top} = \I_d\}.
\end{equation}
Running projected gradient method on this objective function~\eqref{def:BM} definitely saves a large amount of computational resources and memory storage. However, the major issue here is the nonconvexity\footnote{Here we are actually referring to the nonconvexity of $-f(\BS).$ } of the objective function, i.e., there may exist multiple local maximizers in~\eqref{def:BM} and random initialization may lead to one of the local maximizers instead of converging to the global one. As a result, our second focus of this paper is to understand when the Burer-Monteiro approach works for $O(d)$ synchronization. In particular, we are interested in the optimization landscape of $f(\BS)$: 
\emph{when does there exist only one global maximizer, without any other spurious local maximizers?  Moreover, is this global maximizer exactly rank-$d$, i.e.,  it matches the solution to~\eqref{def:OD}?}

\section{Main theorem}\label{s:main}

\subsection{Tightness of SDP relaxation}
Here is our main theorem which provides a deterministic condition to ensure the tightness of SDP relaxation in orthogonal group synchronization.
\begin{theorem}[Deterministic condition for the tightness of SDP relaxation]\label{thm:cvx}
The solutions to~\eqref{def:OD} and~\eqref{def:CVX} are exactly the same, i.e., the global maximizer to the SDP is unique and exactly rank-$d$, if
\[
n \geq  \frac{3\delta^2d\|\BDelta\|^2_{\op}}{2n} + \delta\sqrt{\frac{d}{n}} \|\BDelta\|_{\op}\max_{1\leq i\leq n}\|\BDelta_i\|_{\op}  + \max_{1\leq i\leq n}\left\|\BDelta_{i} ^{\top}\BG\right\|_{\op} + \|\BDelta\|_{\op}, \quad \delta=4,
\]
where $\BDelta_i^{\top} = [\BDelta_{i1}, \cdots,\BDelta_{in}]\in\RR^{d\times nd}$ is the $i$th  block row of $\BDelta$ and $\BDelta_i^{\top}\BG = \sum_{j=1}^n \BDelta_{ij}\BG_j.$
\end{theorem}
Theorem~\ref{thm:cvx} indicates that solving the SDP relaxation yields the global maximizer to~\eqref{def:OD} which is NP-hard in general, under the condition that the noise strength $\|\BDelta\|_{\op}$ is small. 
 Moreover, this condition is purely deterministic and thus it can be easily applied to $O(d)$ synchronization under other noise models. Here we provide one such example under Gaussian random noise.

\begin{theorem}[Tightness of recovery under Gaussian noise]\label{thm:CVX_Gaussian}
The solution to the SDP relaxation~\eqref{def:CVX} is exactly rank-$d$  with high probability if
\[
\sigma \leq \frac{ C_0n^{1/4}}{ d^{3/4}}
\]
for some small constant $C_0$.
\end{theorem}

Our result improves the bound on $\sigma$ by a factor of $\sqrt{d}$,  compared with the recent result by Zhang~\cite{Z19} in which the tightness of SDP holds if
\[
\sigma \lesssim \frac{n^{1/4}}{d^{5/4}}.
\]
One natural question is whether the bound shown above is optimal. The answer is negative. 
Take the case with Gaussian noise as an example: the strength of the planted signal $\BG\BG^{\top}$ is $\|\BG\BG^{\top}\|_{\op} = n.$ The operator norm of the noise is
\[
\|\BDelta\|_{\op} = \sigma\|\BW\|_{\op}= 2\sigma\sqrt{nd}(1 + o(1))
\]
where $\|\BW\|_{\op} = 2\sqrt{nd}(1 + o(1))$ is a classical result for symmetric Gaussian random matrix. For this matrix spike model, the detection threshold should be 
\[
2\sigma\sqrt{nd} \leq n \Longleftrightarrow \sigma \leq \frac{1}{2}\sqrt{\frac{n}{d}}.
\]
In fact, our numerical experiments in Section~\ref{s:numerics} confirm this threshold. This indicates that our analysis still has a large room for improvement: namely improve the dependence of $\sigma$ on $n$ from $n^{1/4}$ to $n^{1/2}.$ 

In particular, if $d=1$, i.e., $\mathbb{Z}_2$-synchronization, SDP is proven to be tight in~\cite{B18} if $\sigma < \sqrt{\frac{n}{(2+\eps)\log n}}$ under Gaussian noise. If $d=2$ and 
\[ 
\BW_{ij} = \frac{1}{\sqrt{2}}
\begin{bmatrix}
X_{ij} & Y_{ij} \\
-Y_{ij} & X_{ij}
\end{bmatrix}
\] 
where $X_{ij}$ and $Y_{ij}$ are independent standard normal, then the model is equivalent to the angular synchronization under complex normal noise which is discussed in~\cite{BBS17,ZB18}. It is shown in~\cite{ZB18} that the factor $n^{1/4}$ can be improved to $n^{1/2}$ by using the leave-one-out technique. We leave the tightness analysis of the orthogonal group synchronization as a future research topic.

\subsection{Optimization landscape of Burer-Monteiro approach}
Our second main result characterizes the optimization landscape of~\eqref{def:BM}. 
\begin{theorem}[Uniqueness and tightness of local maximizer]\label{thm:bm}
For the objective function $f(\BS)$ defined in~\eqref{def:BM}, it has a unique local maximizer which is also the global maximizer 
if $p \geq 2d+1$ and
\[
n \geq  \frac{3\delta^2d\|\BDelta\|^2_{\op}}{2n} + \delta\sqrt{\frac{d}{n}} \|\BDelta\|_{\op}\max_{1\leq i\leq n}\|\BDelta_i\|_{\op}  + \max_{1\leq i\leq n}\left\|\BDelta_{i}^{\top} \BG\right\|_{\op} + \|\BDelta\|_{\op}
\]
where 
\begin{equation}\label{def:gamma}
\delta =  \frac{(2+\sqrt{5})(p+ d)\gamma}{p-2d}, \quad \gamma : = \frac{\|\Tr_d(\BDelta)\|_{\op}}{\|\BDelta\|_{\op}} \vee 1
\end{equation}
and $\Tr_d(\BDelta) = [\Tr(\BDelta_{ij})]_{ij}\in\RR^{n\times n}$ denotes the partial trace of $\BDelta.$ Moreover, this global maximizer is exactly rank-$d$.
\end{theorem}
Theorem~\ref{thm:bm} conveys two messages: one is that the optimization landscape of~\eqref{def:BM} is benign, meaning there exists a unique local maximizer which also corresponds to the global maximizer to the SDP; moreover, it is rank-$d$, indicating the tightness of the global maximizer. The characterization of benign optimization landscape justifies the remarkable performance of the Burer-Monteiro approach. 
\begin{theorem}[Optimization landscape under Gaussian noise]\label{thm:BM_Gaussian}
The optimization landscape of $f(\BS)$ is benign if $p \geq 2d+1$ and
\[
\sigma \leq  C_0\sqrt{\frac{p-2d}{p+d}}\cdot \frac{n^{1/4}}{d^{3/4}}
\]
 with high probability for some small constant $C_0$. In other words, $f(\BS)$ in~\eqref{def:BM} has only one local maximizer which is also global and corresponds to the maximizer of the SDP relaxation~\eqref{def:CVX} and~\eqref{def:OD}. 
\end{theorem}
Similar to the scenario in Theorem~\ref{thm:CVX_Gaussian}, this bound is suboptimal in $n$. The numerical experiments indicate that $\sigma < 2^{-1}n^{1/2}d^{-1/2}$ should be the optimal scaling. However, it remains one major open problem to prove that the landscape is benign for $\sigma$ up to the order $n^{1/2}$, even in the scenario of the angular synchronization~\cite{X19,ZB18}. 
Regarding the choice of $p$, one always wants to keep $p$ as small as possible. 
We believe the bound can be improved to $p \geq 2d$ or even to $p \geq d+2$ from our current bound $p \geq 2d+1$ with more careful analyses. In our numerical experiments, we have seen that $p = 2d$ suffices to ensure global convergence of the generalized power method from any random initialization.

\section{Numerics}\label{s:numerics}

\subsection{Synchronization with Gaussian noise}

Our first experiment is to test how the tightness of~\eqref{def:CVX} depends on the noise strength. 
Consider the group synchronization problem, 
\[
\BA_{ij} = \I_d + \sigma \BW_{ij}\in\RR^{d\times d}, \quad \BA = \BZ\BZ^{\top} + \sigma\BW
\]
where $\BZ^{\top} = [\I_d, \cdots, \I_d]\in\RR^{d\times nd}$ and  $\BW\in\RR^{nd\times nd}$ is a symmetric Gaussian random matrix.
Since solving~\eqref{def:CVX} is rather expensive, we will take an alternative way to find the SDP solution. We first use thd projected power method to get a candidate solution and then confirm it is the global maximizer of the SDP relaxation~\eqref{def:CVX} by verifying the global optimality condition. 

Proposition~\ref{prop:dual} indicates that $\widehat{\BR}\in\RR^{nd\times d}$ is the unique global optimal solution to~\eqref{def:OD} and~\eqref{def:CVX} if
\[
(\widehat{\BLambda} - \BA)\widehat{\BR} = 0, \qquad \lambda_{d+1}(\widehat{\BLambda} - \BA) > 0
\]   
where $\widehat{\BLambda}$ is an $nd\times nd$ block-diagonal matrix with its $i$th block equal to 
\[
\widehat{\BLambda}_{ii} = \frac{1}{2}\sum_{j=1}^n(\widehat{\BR}_i\widehat{\BR}_j^{\top} \BA_{ji} + \BA_{ij}\widehat{\BR}_j\widehat{\BR}_i^{\top} ).
\]

We employ the following generalized projected power iteration scheme:
\begin{equation}\label{def:P}
\BR_i^{(t+1)} = \PP\left(\sum_{j=1}^n \BA_{ij}\BR_j^{(t)} \right) = \BU_i^{(t)}(\BV_i^{(t)})^{\top}
\end{equation}
where the initialization is chosen as $\BR^{(0)} = \BZ$, i.e., $\BR^{(0)}_i = \I_d$.  
Here $\BU_i^{(t)}$ and $\BV_i^{(t)}$ are the left/right singular vectors of $\sum_{j=1}^n  \BA_{ij}\BR_j^{(t)}$ respectively. More precisely, we have
\begin{align*}
\sum_{j=1}^n  \BA_{ij} \BR_j^{(t)}& = \BU_i^{(t)}\BSigma_i^{(t)}(\BV_i^{(t)})^{\top} \\
& = \BU_i^{(t)} \BSigma_i^{(t)} (\BU_i^{(t)})^{\top} \cdot \BU_i^{(t)}   (\BV_i^{(t)})^{\top} \\
& =  \BLambda_{ii}^{(t+1)}\BR_i^{(t+1)}
\end{align*}
and $\BLambda_{ii}^{(t)} = \BU_i^{(t)}\BSigma_i^{(t)}(\BU_i^{(t)})^{\top}$ is a symmetric matrix.
The fixed point $\BR^{(\infty)}$ of this iteration satisfies
\[
\sum_{j=1}^n\BA_{ij}\BR_j^{(\infty)} = \BLambda_{ii}^{(\infty)}\BR_i^{(\infty)}, \quad 1\leq i\leq n
\]
which is actually the first-order necessary condition as discussed in Lemma~\ref{lem:firstcond}.
If the fixed point is found, it remains to show $\lambda_{d+1}(\BLambda^{\infty} - \BA)> 0$ and $(\BLambda^{\infty} - \BA) \BR^{(\infty)} = 0$ in order to confirm $\BR^{(\infty)}(\BR^{(\infty)})^{\top}$ is the optimal solution to~\eqref{def:CVX}. The iteration stops when
\begin{equation}\label{eq:stop}
\| (\BLambda^{(t)} - \BA)\BR^{(t)} \|_{\op} < 10^{-6}, \quad \lambda_{d+1}(\BLambda^{(t)} - \BA) > 10^{-8}
\end{equation}
or the number of iteration reaches 500.

\begin{figure}[h!]
\begin{minipage}{0.48\textwidth}
\centering
\includegraphics[width=82mm]{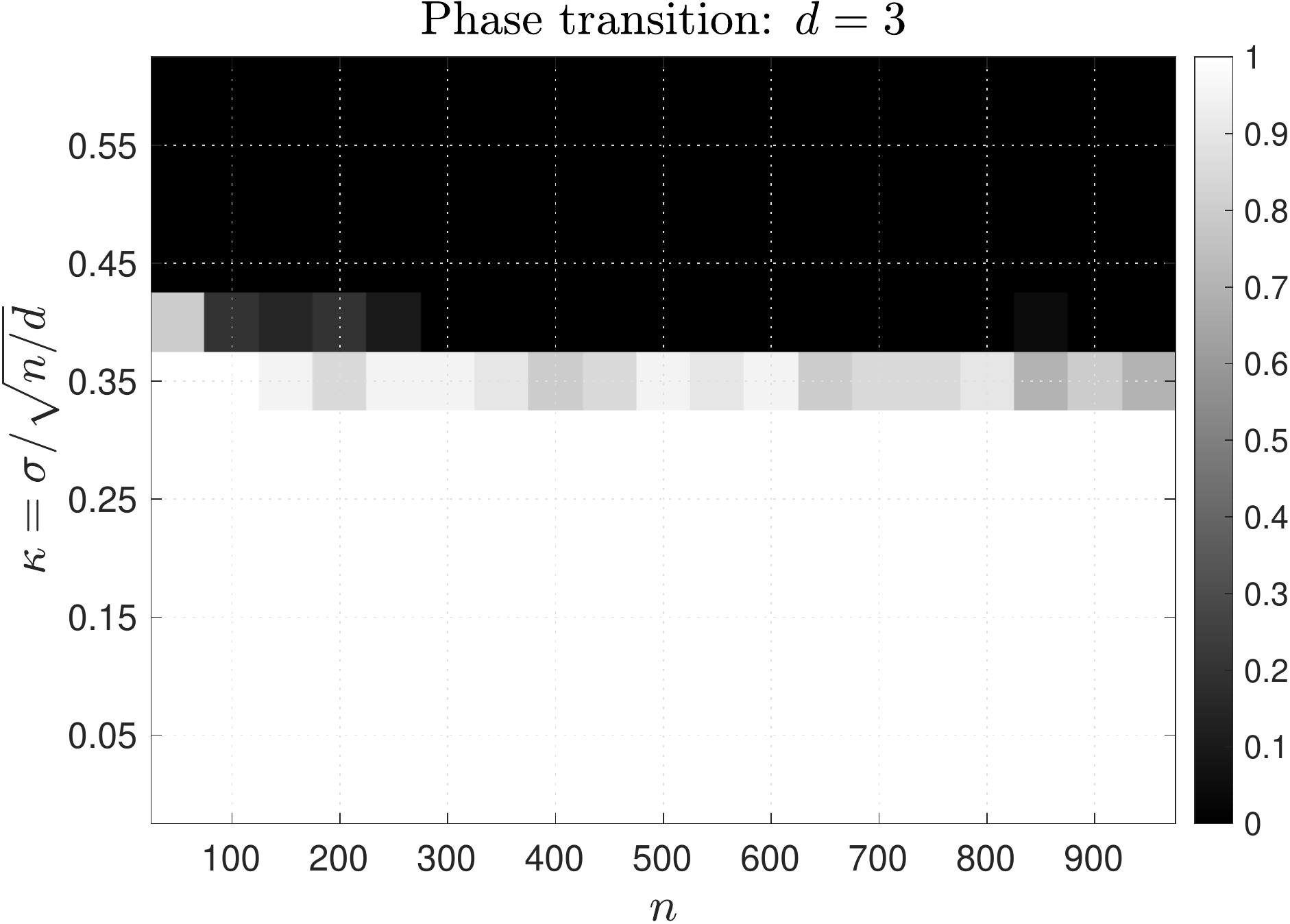}
\end{minipage}
\hfill
\begin{minipage}{0.48\textwidth}
\centering
\includegraphics[width=82mm]{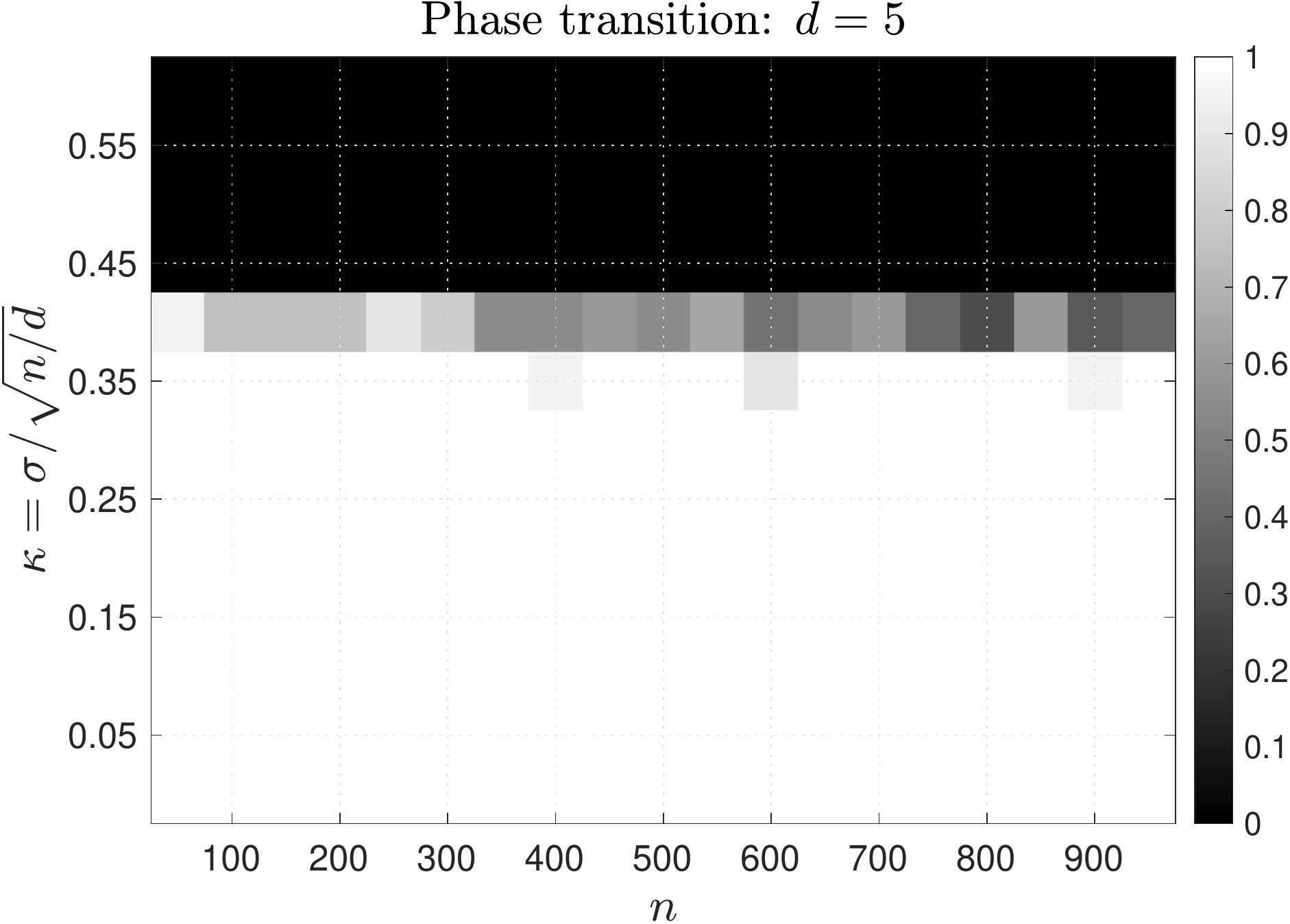}
\end{minipage}
\caption{The phase transition plot for $d = 3$ and $5$. }
\label{fig:PR_Gaussian}
\end{figure}

In this experiment, we let $d=3$ or 5 and $\sigma = \kappa \sqrt{\frac{n}{d}}$. The parameters $(\kappa,n)$ are set to be $0\leq \kappa\leq 0.6$ and $100\leq n\leq 1000$. For each pair of $(\kappa, n)$, we run 20 instances and calculate the proportion of successful cases. 
From Figure~\ref{fig:PR_Gaussian}, we see that if $\kappa < 0.35$, the SDP is tight, i.e., it recovers the global minimizer to the least squares objective function.
The phase transition plot does not depend heavily on the parameter $d$. This confirms our conjecture that $\kappa < 1/2$ (modulo a log factor), the SDP relaxation is tight.

\subsection{Phase transition plot for nonconvex low-rank optimization}
Instead of applying Riemannian gradient method to~\eqref{def:BM}, we employ projected power method to show how the convergence depends on the noise level. 
Here the power method is viewed as projected gradient ascent method. We~\emph{randomly} initialize each $\BS_i^{(0)}$ by creating a $d\times p$ Gaussian random matrix and extracting the random row space via QR decomposition. Then we perform
\[
\BS_i^{(t+1)} = \PP\left( \sum_{j=1}^n\BA_{ij} \BS_j^{(t)} \right)
\]
and the projection operator $\PP$ is defined in~\eqref{def:P}.

\begin{figure}[h!]
\begin{minipage}{0.48\textwidth}
\centering
\includegraphics[width=82mm]{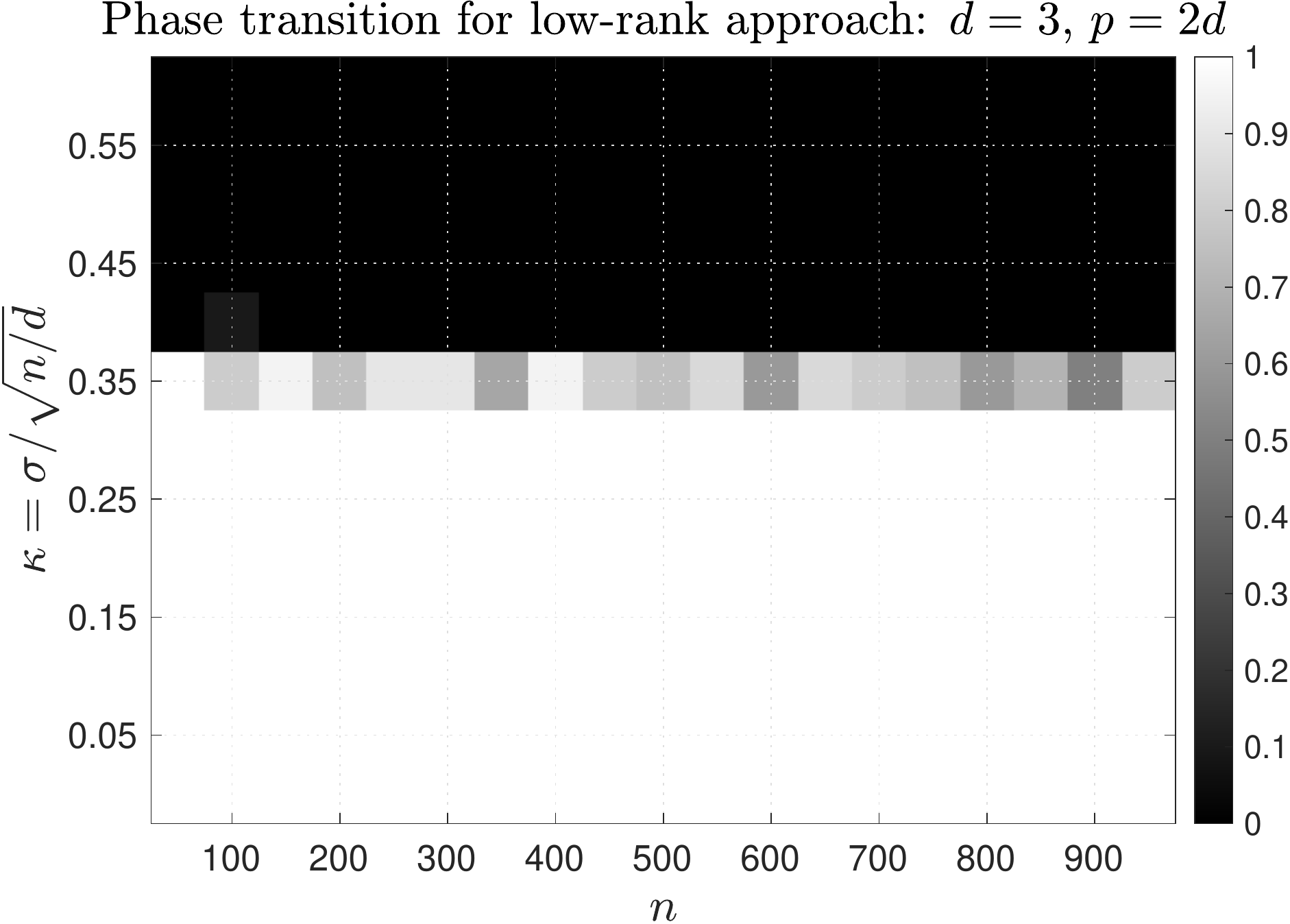}
\end{minipage}
\hfill
\begin{minipage}{0.48\textwidth}
\centering
\includegraphics[width=82mm]{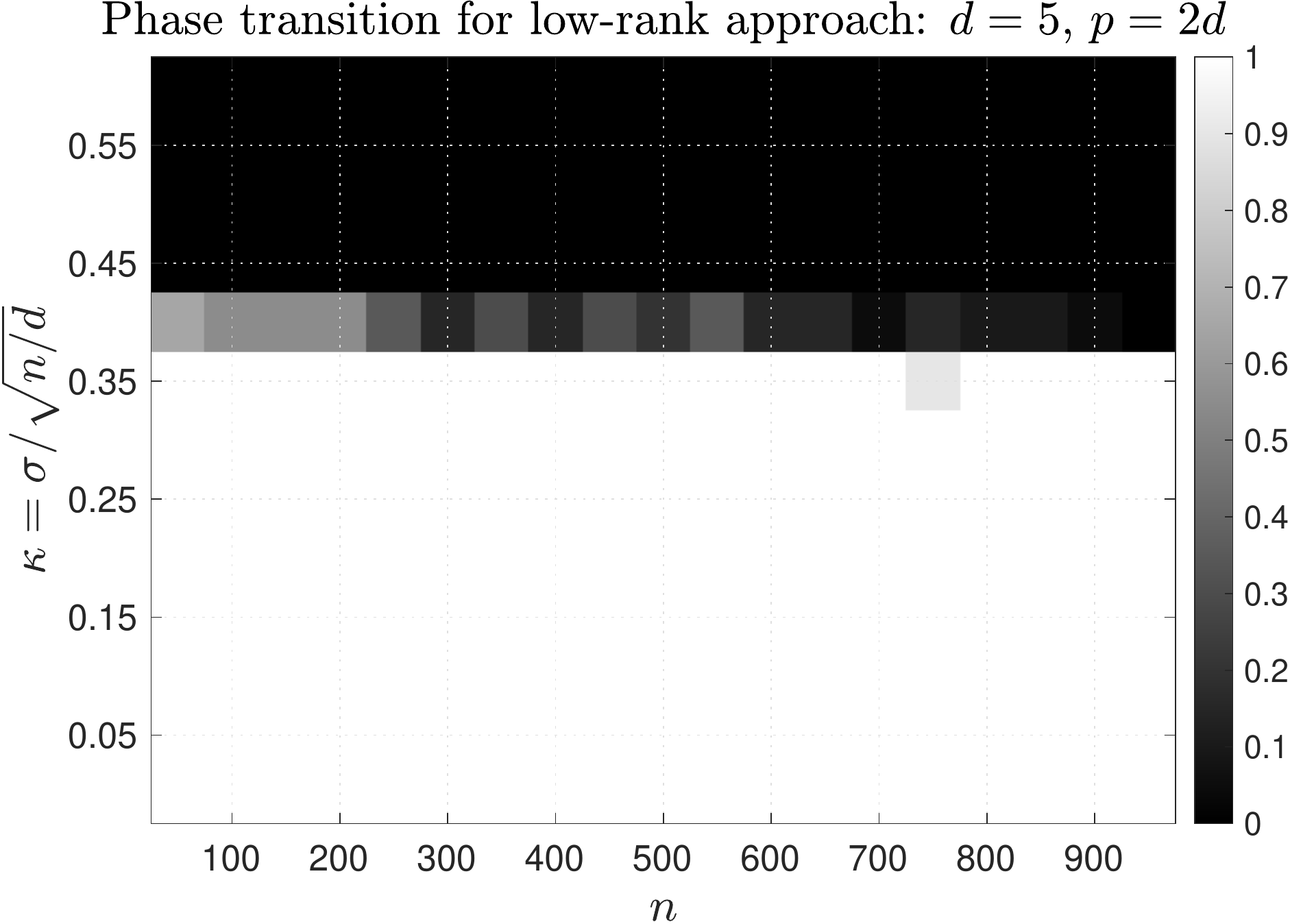}
\end{minipage}
\caption{The phase transition plot for $d = 3$ and $5$. }
\label{fig:PR_BM_Gaussian}
\end{figure}
After the iterates stabilize, we use~\eqref{eq:stop} to verify the global optimality and tightness of the solution. Here we set $ p =2d$ and for each pair of $\kappa$ and $n$, we still run 20 experiments and calculate the proportion of successful instances. Compared with Figure~\ref{fig:PR_Gaussian}, Figure~\ref{fig:PR_BM_Gaussian} provides highly similar phase transition plots for both $d = 3$ and $d = 5$. This is a strong indicator that the objective function is likely to have a benign landscape even if $\sigma = \Omega(\sqrt{nd^{-1}})$, which is much more optimistic than our current theoretical bound.

\section{Proofs}\label{s:proof}

\subsection{The roadmap of the proof}

The proof consists of several sections and some parts are rather technical. Thus we provide a roadmap of the proof here.
For both the analysis of convex relaxation and the Burer-Monteiro factorization, the key is to analyze the objective function $f(\BS)$ defined in~\eqref{def:BM} for $\BS\in\St(d,p)^{\otimes n}$. Our analysis consists of several steps:
\begin{enumerate}[(a)]
\item For~\eqref{def:BM}, we first provide a sufficient condition to certify the global optimality and tightness of $\BS$ by using the duality theory in convex optimization. This is given in Proposition~\ref{prop:dual}.
\item We show that $\BS$ is the global maximizer of~\eqref{def:BM} if $\BS$ is a second-order critical point and is sufficiently close to the~\emph{fully synchronized state}, i.e., $\BS_i = \BS_j,\forall i\neq j$. Moreover, the rank of $\BS$ equals $d$ and thus it is tight. This leads to Proposition~\ref{prop:keys}.

\item For convex relaxation, we show that the global maximizer of~\eqref{def:BM} must be highly aligned with the fully synchronized state, see Proposition~\ref{prop:SZ}; for the Burer-Monteiro approach, we prove that~\emph{all} the second order critical points (SOCP) of~\eqref{def:BM} must be close to the fully synchronized state, as shown in the Proposition~\ref{prop:SZBM}. 

\item Combining all the supporting results together finishes the proof.
\end{enumerate}

The idea of proof is mainly inspired by~\cite{BBS17,BBV16} which focus on the $\mathbb{Z}_2$- and angular synchronization. However, due to the non-commutativity of $O(d)$ for $d\geq 3$, several parts require quite different treatments. 
Now we present the first proposition which gives a sufficient condition to guarantee the global optimality and tightness, and establish the equivalence of the global maximizers among the three optimization programs~\eqref{def:OD},~\eqref{def:CVX}, and~\eqref{def:BM}.  Without loss of generality, we assume $\BA = \BZ\BZ^{\top} + \BDelta$, i.e., $\BA_{ij} = \I_d + \BDelta_{ij}$, where $\BZ^{\top} = [\I_d, \cdots, \I_d]\in\RR^{d\times nd}$ from now on.
\begin{proposition}\label{prop:dual}
Let $\BLambda$ be an $nd\times nd$ block diagonal matrix  $\BLambda = \blkdiag(\BLambda_{11},\cdots,\BLambda_{nn})$. Suppose $\BLambda$ satisfies 
\begin{equation}\label{eq:dual}
 (\BLambda - \BA)\BS = 0, \quad \BLambda - \BA\succeq 0, 
\end{equation}
for some $\BS\in\St(d,p)^{\otimes n}$, then $\BS\BS^{\top}\in \RR^{nd\times nd}$ is a global optimal solution to the SDP relaxation~\eqref{def:CVX}.
Moreover, $\BS$ is the unique global optimal solution to~\eqref{def:OD} if the following additional rank assumption holds
\begin{equation}\label{eq:rank}
\rank(\BLambda - \BA) = (n-1)d.
\end{equation}
\end{proposition}
The condition~\eqref{eq:dual} provides a sufficient condition for $\BX = \BS\BS^{\top}$ to be one global maximizer to~\eqref{def:CVX} and~\eqref{def:BM}. The condition~\eqref{eq:rank} characterizes when the solution $\BS\in\RR^{p\times nd}$ is of rank $d$ and unique. In particular, if $\rank(\BS) = d$, then $\BS$ is actually the global maximizer to~\eqref{def:OD}.

\vskip0.25cm

The next step is to show all the second order critical points, i.e., those points whose Riemannian gradient equals 0 and Hessian is positive semidefinite, are actually global maximizers if they are close to the fully synchronized state. It suffices to show that those SOCPs satisfy the global optimality condition~\eqref{eq:dual} and~\eqref{eq:rank}. In fact, if $\BS$ is a first order critical point, we immediately have $(\BLambda - \BA) \BS = 0$ for some block-diagonal matrix $\BLambda\in\RR^{nd\times nd}.$
\begin{lemma}[First order critical point]\label{lem:firstcond}
The first order critical point of $f(\BS)$ satisfies:
\[
\sum_{j=1}^n \BA_{ij}\BS_j = \BLambda_{ii}\BS_i, \quad \BS_i\in\St(d,p)
\]
where $\BLambda_{ii}$ equals
\begin{equation}\label{def:Lambda}
\BLambda_{ii} := \frac{1}{2}\sum_{j=1}^n\left(\BS_i\BS_j^{\top}\BA_{ji} + \BA_{ij}\BS_j\BS_i^{\top}\right).
\end{equation}
\end{lemma}
The proof of Lemma~\ref{lem:firstcond} is given in Section~\ref{s:gradhess}. Lemma~\ref{lem:firstcond} shows that $\BLambda_{ii}$ depends on $\BS$ and is completely determined by $(\BLambda - \BA)\BS = 0$. 
As a result, it suffices to prove that $\BLambda - \BA\succeq 0$ for some second  order critical points which obey the~\emph{proximity condition}, i.e., $\BS$ is sufficiently close to $\BZ.$
To quantify this closeness, we introduce the following distance: given any $\BS\in\St(d,p)^{\otimes n}$, the distance between $\BS$ and the fully synchronized state is defined by
\begin{equation}\label{def:df}
d_F(\BS, \BZ) : = \min_{\BQ\in\RR^{d\times p}, \BQ\BQ^{\top}=\I_d} \| \BS - \BZ\BQ \|_F = \min_{\BQ\in\RR^{d\times p}, \BQ\BQ^{\top}=\I_d} \sqrt{\sum_{i=1}^n \|\BS_i - \BQ\|_F^2}
\end{equation}
where $(\BZ\BQ)^{\top} = [\BQ^{\top}, \cdots,\BQ^{\top}]\in\RR^{p\times nd}$. For the rest of the paper, we will let $\BQ$ be the $d\times p$ partial orthogonal matrix which minimizes~\eqref{def:df}. In fact, the minimizer equals $\BQ = \PP(\BZ^{\top}\BS)$ where $\BZ^{\top}\BS = \sum_{j=1}^n \BS_j\in\RR^{d\times p}$ and $\PP(\cdot)$ is defined in~\eqref{def:P}.

\begin{condition}[\bf Proximity condition]
A feasible solution $\BS\in\RR^{p\times nd}\in\St(d,p)^{\otimes n}$ satisfies the proximity condition if
\begin{equation}\label{eq:delta}
d_F(\BS,\BZ) \leq \delta \sqrt{\frac{d}{n}}\|\BDelta\|_{\op}
\end{equation}
for some constant $\delta > 0$. 
\end{condition}

The next Proposition is the core of the whole proof, stating that any SOCPs satisfying the proximity condition~\eqref{eq:delta} are global maximizers to~\eqref{def:OD} and~\eqref{def:CVX}. 
\begin{proposition}[\bf Global optimality of (P) and (SDP)]\label{prop:keys}
For a second order critical point $\BS$ satisfying~\eqref{eq:delta}, it is the unique global maximizer to both (P) and (SDP) if
\begin{equation}\label{eq:suff}
n \geq \frac{3\delta^2d\|\BDelta\|^2_{\op}}{2n} + \max_{1\leq i\leq n} \left\|  \sum_{j\neq i} \BDelta_{ij}\BS_j\right\|_{\op} + \|\BDelta\|_{\op}
\end{equation}
where
\begin{equation}\label{eq:maxSA}
\max_{1\leq i\leq n}\left\|\sum_{j\neq i} \BDelta_{ij}\BS_j\right\|_{\op} \leq \delta\sqrt{\frac{d}{n}} \|\BDelta\|_{\op}\max_{1\leq i\leq n}\|\BDelta_i\|_{\op}  + \max_{1\leq i\leq n}\left\|\BDelta_{i}^{\top} \BZ\right\|_{\op}.
\end{equation}
In other words, the global optimality of $\BS$ is guaranteed by 
\[
n \geq  \frac{3\delta^2d\|\BDelta\|^2_{\op}}{2n} + \delta\sqrt{\frac{d}{n}} \|\BDelta\|_{\op}\max_{1\leq i\leq n}\|\BDelta_i\|_{\op}  + \max_{1\leq i\leq n}\left\|\BDelta_{i}^{\top} \BZ\right\|_{\op} + \|\BDelta\|_{\op}.
\]
\end{proposition}
\begin{remark}
Proposition~\ref{prop:keys} provides a simple criterion to verify a near-fully synchronized state is the global optimal solution. However, the estimation of $\max_{1\leq i\leq n}\left\|\sum_{j\neq i} \BDelta_{ij}\BS_j\right\|_{\op} $ is not tight which leads to the suboptimal bound in the main theorems. The major difficulty results from the complicated statistical dependence between $\BDelta$ and any second-order critical points $\BS$. This is well worth further investigation for $O(d)$.

\end{remark}
Now we present two propositions which demonstrate that any global maximizers and second-order critical points to~\eqref{def:BM} satisfy~\eqref{eq:delta} for some $\delta > 0$.

\paragraph{(i) Convex relaxation:}
For the tightness of SDP relaxation, we show that the global maximizer to~\eqref{def:OD} must satisfy~\eqref{eq:delta} with $\delta = 4$.

\begin{proposition}[\bf Proximity condition for convex relaxation]\label{prop:SZ}
The global maximizers to~\eqref{def:BM} satisfy
\[
d_F(\BS,\BZ) \leq \delta\sqrt{\frac{d}{n}}\|\BDelta\|_{\op}, \quad \delta = 4.
\]
\end{proposition}
This proposition essentially ensures that any global maximizer to~\eqref{def:BM} is close to the fully synchronized state and its distance depends on the noise strength.

\paragraph{(ii) Low-rank approach:}
For the Burer-Monteiro approach, we prove that if $p\geq 2d+1$, all the local maximizers of~\eqref{def:BM} satisfy~\eqref{eq:delta} with $\delta$ which depends on $p$, $d$, and $\gamma$.
\begin{proposition}[\bf Proximity condition for low-rank approach]\label{prop:SZBM}
Suppose $p \geq 2d+1$. All the second-order critical points $\BS$ of $f(\BS)$ in~\eqref{def:BM} satisfy
\[
d_F(\BS,\BZ) \leq \delta\sqrt{\frac{d}{n}}\|\BDelta\|_{\op}, \quad \delta =  \frac{(2+\sqrt{5})(p+ d)\gamma}{p-2d}
\]
where
\[
\gamma : = \frac{\|\Tr_d(\BDelta)\|_{\op}}{\|\BDelta\|_{\op}} \vee 1
\]
and $\Tr_d(\BDelta) = [\Tr(\BDelta_{ij})]_{ij}\in\RR^{n\times n}$ denotes the partial trace of $\BDelta.$
\end{proposition}
\begin{remark}
If $\BDelta$ is a symmetric Gaussian random matrix, then $\Tr_d(\BDelta)$ is an $n\times n$ Gaussian random matrix whose entry is $\mathcal{N}(0,d)$ and $\|\Tr_d(\BDelta)\|_{\op} = (1 + o(1))\|\BDelta\|_{\op}$ holds.
\end{remark}
\vskip0.25cm

We defer the proof of Proposition~\ref{prop:dual},~\ref{prop:keys},~\ref{prop:SZ} and~\ref{prop:SZBM} to Section~\ref{s:propdual},~\ref{s:propkeys},~\ref{s:propSZ} and~\ref{s:propSZBM} respectively. Now we provide a proof of Theorem~\ref{thm:cvx} and~\ref{thm:bm} by using the aforementioned propositions.

\begin{proof}[\bf Proof of Theorem~\ref{thm:cvx}]
To prove the tightness of convex relaxation, we first consider the global maximizer to~\eqref{def:BM} which is also a second-order critical point. By Proposition~\ref{prop:SZ}, we have $d_F(\BS,\BZ) \leq \delta\sqrt{n^{-1}d}\|\BDelta\|_{\op}$ with $\delta = 4.$  With Proposition~\ref{prop:keys}, we immediately have Theorem~\ref{thm:cvx}.
\end{proof}

\begin{proof}[\bf Proof of Theorem~\ref{thm:bm}]
To analyze the landscape of~\eqref{def:BM}, we invoke Proposition~\ref{prop:SZBM} which states that~\emph{all} the second-order critical points (SOCP) are essentially close to the fully synchronized state. Now it suffices to show that all SOCPs are global maximizers to~\eqref{def:CVX} and~\eqref{def:OD} and the global maximizer is unique under the assumption of Theorem~\ref{thm:bm}. This is fortunately guaranteed by Proposition~\ref{prop:keys}. 
\end{proof}

\subsection{Riemannian gradient and Hessian matrix}\label{s:gradhess}

We start with analyzing the SOCPs of $f(\BS)$ by first computing its Riemannian gradient and Hessian. The calculation involves the tangent space at $\BS_i\in\St(d,p)$ which is given by
\begin{equation}\label{def:Tst}
T_{\BS_i}({\cal M}) := \{\BY_i\in\RR^{d\times p}: \BS_i\BY_i^{\top} + \BY_i\BS_i^{\top} = 0\}, \quad {\cal M} : = \St(d,p).
\end{equation}
In other words, $ \BS_i\BY_i^{\top}$ is an anti-symmetric matrix if $\BY_i\in\RR^{d\times p}$ is an element in the tangent space.

\begin{proof}[\bf Proof of Lemma~\ref{lem:firstcond}]

Recall the objective function $f(\BS) =  \sum_{i=1}^n\sum_{j=1}^n \lag \BS_i,\BA_{ij}\BS_j \rag$ in~\eqref{def:BM}
where $\BS_i\in\St(d,p)$. We take the gradient w.r.t. $\BS_i$ in the Euclidean space.
\[
\frac{\pa f}{\pa \BS_i} = \sum_{j\neq i} \BA_{ij}\BS_j.
\]
The Riemannian gradient w.r.t. $\BS_i$ is
\begin{equation}\label{eq:grad}
\nabla_{\BS_i}f  = \text{Proj}_{T_{\BS_i}({\cal M})}\left( \sum_{j\neq i} \BA_{ij}\BS_j\right)  =\sum_{j=1}^n \BA_{ij}\BS_j - \frac{1}{2}\left(  \sum_{j=1}^n\BS_i\BS_j^{\top}\BA_{ji} + \BA_{ij}\BS_j\BS_i^{\top} \right)\BS_i
\end{equation}
by projecting $\frac{\pa f}{\pa \BS_i}$ onto the tangent space $T_{\BS_i}({\cal M})$ at $\BS$, as shown in~\cite[Equation (3.35)]{AMS09}:
\[
\text{Proj}_{T_{\BS_i}({\cal M})}\left(\BPi\right) = \BPi - \frac{1}{2}\left(\BPi\BS_i^ {\top} + \BS_i\BPi^{\top}\right)\BS_i
\]
where $\BPi\in\RR^{d\times p}$ and the matrix manifold ${\cal M}$  is $\St(d,p).$

Setting $\nabla_{\BS_i}f  = 0$ gives $\BLambda_{ii}$ in~\eqref{def:Lambda} and
\[
\sum_{j=1}^n\BA_{ij}\BS_j  = \BLambda_{ii}\BS_i, 
\]
In other words, 
$(\BLambda - \BA) \BS = 0$
where $\BLambda = \text{blkdiag}(\BLambda_{11}, \cdots,\BLambda_{nn}).$
\end{proof}

Next, we compute the Riemannian Hessian and prove that  $\BLambda\succeq 0$ for any second order critical point. 

\begin{lemma}\label{lem:Hess}
The quadratic form associated to the Hessian matrix of~\eqref{def:BM} is
\[
\dot{\BS} : \nabla^2_{\pa \BS\pa\BS} f(\BS) :\dot{\BS} = - \sum_{i=1}^n \lag \BLambda_{ii}, \dot{\BS}_i\dot{\BS}_i^{\top} \rag + \sum_{i=1}^n\sum_{j=1}^n \lag\BA_{ij}, \dot{\BS}_i\dot{\BS}_j^{\top}\rag
\]
where $\dot{\BS}^{\top} = [\dot{\BS}_1^{\top},\cdots,\dot{\BS}_n^{\top}]\in\RR^{p\times nd}$ and $\dot{\BS}_i\in\RR^{d\times p}$ is an element on the tangent space of Stiefel manifold at $\BS_i.$ In other words, if $\BS$ is a second order critical point, it must satisfy:
\begin{equation}\label{cond:2ndpsd}
\sum_{i=1}^n \lag \BLambda_{ii}, \dot{\BS}_i\dot{\BS}_i^{\top}\rag \geq  \sum_{i=1}^n\sum_{j=1}^n\lag\BA_{ij}, \dot{\BS}_i\dot{\BS}_j^{\top}\rag.
\end{equation}
\end{lemma}
\begin{proof}
Recall the Riemannian gradient w.r.t. $\BS_i$ is given by
\begin{align*}
\nabla_{\BS_i}f &  = \sum_{j\neq i} \BA_{ij}\BS_j - \frac{1}{2}\left(  \sum_{j\neq i}  \BS_i\BS_j^{\top}\BA_{ji} + \BA_{ij}\BS_j\BS_i^{\top}\right)\BS_i.
\end{align*}
Let $\dot{\BS}_i$ be a matrix on the tangent space at $\BS_i$:
\[
\lim_{t\rightarrow 0}\frac{\nabla_{\BS_i}f(\BS + t\dot{\BS}_i) - \nabla_{\BS_i}f(\BS) }{t} = -( \BLambda_{ii} -\I_d) \dot{\BS}_i  - \frac{1}{2} \sum_{j\neq i}(\dot{\BS}_i\BS_j^{\top}\BA_{ji} + \BA_{ij}\BS_j\dot{\BS}_i^{\top})\BS_i
\]
where $(\BS+t\dot{\BS}_{i})^{\top} = [\BS_1^{\top}, \cdots, (\BS_i + t\dot{\BS}_i)^{\top}, \cdots, \BS_n^{\top}]$ and $\BS_i\BS_i^{\top} = \I_d.$ 
As a result, the quadratic form associated to the Riemannian Hessian  is
\begin{align*}
-\dot{\BS}_i : \nabla^2_{\pa \BS_i\pa \BS_i} f(\BS) : \dot{\BS}_i 
& = \left\lag (\BLambda_{ii} -\I_d)\dot{\BS}_i  + \frac{1}{2}\sum_{j\neq i} (\dot{\BS}_i\BS_j^{\top} \BA_{ji} + \BA_{ij}\BS_j\dot{\BS}_i^{\top} )\BS_i, \dot{\BS}_i \right\rag \\
& = \lag \BLambda_{ii} -\I_d , \dot{\BS}_i\dot{\BS}_i^{\top} \rag 
+ \frac{1}{2} \sum_{j\neq i} \lag \dot{\BS}_i\BS_j^{\top}\BA_{ji}, \dot{\BS}_i\BS_i^{\top}\rag 
+ \frac{1}{2} \sum_{j\neq i}\lag \BA_{ij}\BS_j\dot{\BS}_i^{\top}, \dot{\BS}_i\BS_i^{\top}\rag \\
& = \lag \BLambda_{ii} -\I_d, \dot{\BS}_i\dot{\BS}_i^{\top} \rag 
+ \frac{1}{2} \sum_{j\neq i} \lag \BA_{ij}\BS_j\dot{\BS}_i^{\top}, \BS_i\dot{\BS}_i^{\top} +\dot{\BS}_i\BS_i^{\top} 
\rag  \\
& = \lag \BLambda_{ii} -\I_d, \dot{\BS}_i\dot{\BS}_i^{\top}\rag 
\end{align*}
where $\BS_i\dot{\BS}_i^{\top} +\dot{\BS}_i\BS_i^{\top}  = 0$ since $\dot{\BS}_i$ is on $T_{\BS_i}({\cal M})$.

For the mixed partial derivative, we have
\[
\lim_{t\rightarrow 0} \frac{\nabla_{\BS_i}f(\BS+t\dot{\BS}_j) - \nabla_{\BS_i}f(\BS)}{t} = \BA_{ij}\dot{\BS_j} - \frac{1}{2}\left( \BS_i\dot{\BS}_j^{\top} \BA_{ji} + \BA_{ij}\dot{\BS}_j\BS_i^{\top}\right)\BS_i
\]
for $j\neq i$.
Thus
\begin{align*}
\dot{\BS}_i : \nabla^2_{\pa \BS_i\pa \BS_j}f(\BS) :\dot{\BS}_j & = \left\lag \BA_{ij}\dot{\BS_j} - \frac{1}{2}\left( \BS_i\dot{\BS}_j^{\top} \BA_{ji} + \BA_{ij}\dot{\BS}_j\BS_i^{\top}\right)\BS_i , \dot{\BS}_i\right\rag \\
& = \lag \BA_{ij}\dot{\BS}_j, \dot{\BS}_i\rag - \frac{1}{2}\lag  \BS_i\dot{\BS}_j^{\top}  \BA_{ji} + \BA_{ij}\dot{\BS}_j\BS_i^{\top}, \dot{\BS}_i \BS_i^{\top} \rag \\
& = \lag \BA_{ij}\dot{\BS}_j, \dot{\BS}_i\rag - \frac{1}{2}\lag  \BA_{ij}\dot{\BS}_j\BS_i^{\top}, \BS_i\dot{\BS}_i^{\top} + \dot{\BS}_i\BS_i^{\top} \rag \\
& =  \lag \BA_{ij}\dot{\BS}_j, \dot{\BS}_i\rag.
\end{align*}
Taking the sum of $\dot{\BS}_i : \nabla^2_{\pa \BS_i\pa \BS_j}f(\BS) :\dot{\BS}_j $ over $(i,j)$ gives
\[
\dot{\BS}: \nabla^2_{\pa \BS\pa \BS}f : \dot{\BS} = -\sum_{i=1}^n\lag \BLambda_{ii}, \dot{\BS}_i\dot{\BS}_i^{\top}\rag + \sum_{i=1}^n\sum_{j=1}^n \lag \BA_{ij}, \dot{\BS}_i\dot{\BS}_j^{\top} \rag.
\]
If $\BS$ is a local maximizer of~\eqref{def:BM}, then $\dot{\BS}: \nabla^2_{\pa \BS\pa \BS}f : \dot{\BS} \leq 0$ holds for any $\dot{\BS}\in (T_{\BS_i}({\cal M}))^{\otimes n}.$
\end{proof}

Suppose $\BS$ is a local maximizer of~\eqref{def:BM}, then~\eqref{cond:2ndpsd} implies that
\[
\lag \BLambda_{ii} - \I_d, \dot{\BS}_i\dot{\BS}_i^{\top} \rag  \geq 0, \quad \dot{\BS}_i\in T_{\BS_i}({\cal M}).
\]
Does it imply that $\BLambda_{ii} \succeq \I_d$? The answer is yes if $p > d$. However, this is not longer true if $p = d.$ For $p = d$, we are only able to prove that the sum of the smallest two eigenvalues is nonnegative. 

\begin{lemma}\label{lem:Lambdapos}
Suppose $\BS$ is a local maximizer, then it holds 
\[
\BLambda_{ii} \succeq \I_d, \quad 1\leq i\leq n.
\]
\end{lemma}

\begin{proof}
Note that $\BS_i\in\RR^{d\times p}$ with $p>d$ is a ``fat" matrix. It means we can always find $\bv_i\in\RR^p$ which is perpendicular to all rows of $\BS_i$, i.e., $\BS_i\bv_i = 0$. Without loss of generality, we assume $\bv_i$ is a unit vector. Now we construct $\dot{\BS}_i$ in the following form:
\[
\dot{\BS}_i =\bu_i \bv_i^{\top}
\]
where $\bu_i$ is an arbitrary vector in $\RR^d.$
It is easy to verify that
\[
\BS_i\dot{\BS}_i^{\top}  = \BS_i  \bv_i \bu_i^{\top} = 0,
\]
which means $\dot{\BS}_i$ is indeed an element in the tangent space of $\St(d,p)$ at $\BS_i.$

Now, we have
\[
\lag \BLambda_{ii} - \I_d, \bu_i \bu_i^{\top}  \rag = \lag \BLambda_{ii} - \I_d, \bu_i \bv_i^{\top}\bv_i \bu_i^{\top}  \rag = \lag \BLambda_{ii} - \I_d, \dot{\BS}_i\dot{\BS}_i^{\top} \rag  \geq 0, \quad \forall \bu_i\in\RR^d
\]
which implies that $\BLambda_{ii}- \I_d \succeq 0$ and $\BLambda_{ii}\succeq \I_d.$
\end{proof}

\subsection{Certifying global optimality via dual certificate}\label{s:propdual}
To guarantee the global optimality of a feasible solution, we will employ the standard tools from the literature in compressive sensing and low-rank matrix recovery. The core part is to construct the dual certificate which confirms that the proposed feasible solution and the dual certificate yield strong duality.

\begin{proof}[\bf Proof of Proposition~\ref{prop:dual}]
We start from the convex optimization~\eqref{def:CVX} and derive its dual program.
First introduce the symmetric matrix $\BPi_{ii}\in\RR^{d\times d}$ as the dual variable corresponding to the constraint $\BX_{ii} = \I_d$ and then get the Lagrangian function. Here we switch from maximization to minimization in~\eqref{def:CVX} by changing the sign in the objective function.
\begin{align*}
{\cal L}(\BX, \BPi) & = \sum_{i=1}^n \lag \BPi_{ii}, \BX_{ii} - \I_d\rag - \lag \BA, \BX\rag \\
& = \lag \BPi - \BA, \BX\rag - \Tr(\BPi)
\end{align*}
where $\BPi = \blkdiag(\BPi_{11}, \cdots, \BPi_{nn})\in\RR^{nd\times nd}$ and $\BX\succeq 0$. If $\BPi - \BA$ is not positive semidefinite, taking the infimum w.r.t. $\BX\succeq 0$ for the Lagrangian function gives negative infinity. Thus we require $\BPi - \BA\succeq 0$:
\[
\inf_{\BX\succeq 0} {\cal L}(\BX, \BPi)  = -\Tr(\BPi).
\]
As a result, the dual program of~\eqref{def:CVX} is equivalent to
\[
\min_{\BPi\in\RR^{nd\times nd}} \Tr(\BPi) \quad \text{such that} \quad \BPi - \BA\succeq 0, \quad \BPi \text{ is block-diagonal.}
\]
Weak duality in convex optimization~\cite{BN01} implies that $\Tr(\BPi) \geq \lag \BA,\BX\rag$. Moreover, 
$(\BX, \BPi)$ is a primal-dual optimal solution (not necessarily unique) if the complementary slackness holds
\begin{equation}\label{eq:sdual1}
\lag \BPi- \BA, \BX\rag = 0, \quad \BPi - \BA\succeq 0
\end{equation}
since~\eqref{eq:sdual1} implies strong duality, i.e., $\Tr(\BPi) = \lag \BA, \BX\rag$ since $\BX_{ii} = \I_d.$
In fact, this condition~\eqref{eq:sdual1} is equivalent to
\begin{equation}\label{eq:sdual}
(\BPi- \BA) \BX= 0, \quad \BPi - \BA\succeq 0
\end{equation}
because both $\BPi- \BA $ and $\BX$ are positive semidefinite.

\vskip0.5cm

Let $\BS\in\RR^{nd\times p}$ be a  feasible solution. Suppose there exists an $nd\times nd$ block diagonal matrix $\BLambda$ and satisfies~\eqref{eq:dual}. The global optimality of $\BX = \BS\BS^{\top}$ follows directly from~\eqref{eq:dual} and~\eqref{eq:sdual}. In addition, if the rank of $\BPi - \BA$ is $(n - 1)d$, then the global optimizer to~\eqref{def:CVX} is exactly rank-$d$. This is due to $(\BPi - \BA)\BX = 0$, implying that the rank of $\BX$ is at most $d$ but $\BX_{ii} =\I_d$ guarantees $\rank(\BX)\geq d.$ 
This results in the tightness of~\eqref{def:CVX} since the global optimal solution to the SDP is exactly rank-$d$ and thus must be the global optimal solution to (P) as well.

Now we prove that if $\rank(\BPi - \BA) = (n-1)d$, then $\BX$ is the~\emph{unique} maximizer. Let's prove it by contradiction. If not,  then there exists another global maximizer $\widetilde{\BX}$ such that $\lag \BA, \widetilde{\BX} \rag = \Tr(\BPi)=\lag \BPi, \widetilde{\BX}\rag$ since the feasible solution $\BX$ and $\widetilde{\BX}$ achieve the same primal value due to the linearity of the objective function:
\[
\lag \BPi - \BA, \widetilde{\BX}\rag = 0 \Longrightarrow (\BPi - \BA)\widetilde{\BX} = 0. 
\]
Since $\rank(\BPi- \BA) = (n-1)d$, thus $\rank(\widetilde{\BX}) \leq d$. Note that each diagonal block is $\I_d$ and it implies $\rank(\widetilde{\BX}) = d$. This proves that $\widetilde{\BX}= \BX$ holds (modulo a global rotation in the column space) since $\BX$ and $\widetilde{\BX}$ are determined uniquely by the null space of $\BPi - \BA.$ 
\end{proof}

Proposition~\ref{prop:dual} indicates that in order to show that a first-order critical point of $f(\BS)$ is the unique global maximizer to (P), it suffices to guarantee $\BLambda - \BA\succeq 0$ and $\lambda_{d+1}(\BLambda - \BA) > 0$. This is equivalent to $\rank(\BLambda - \BA) = (n-1)d$ here since the first order necessary condition implies $(\BLambda - \BA)\BS = 0$ for $\BLambda$ defined in~\eqref{def:Lambda} which means at least $d$ eigenvalues of $\BLambda - \BA$ are zero.
Now we can see that the key is to ensure $\BLambda - \BA\succeq 0$ for some first-order critical point $\BS$ (i.e., those critical points which satisfy the proximity condition). Define the certificate matrix
\begin{equation}\label{def:C}
\BC := \BLambda  - \BA, \quad \BC_{ij} = 
\begin{cases}
\BLambda_{ii} - \BA_{ii}, \quad i=j, \\
-\BA_{ij}, \quad i\neq j,
\end{cases}
\end{equation}
for any given $\BS.$
From the definition, we know that any first-order critical points satisfy $\BC \BS = 0.$

\subsection{Proof of Proposition~\ref{prop:keys}}\label{s:propkeys}

\begin{lemma}\label{lem:sigmamin}
Suppose the proximity condition~\eqref{eq:delta} holds, we have
\[
n \geq \sigma_{\max}(\BZ^{\top}\BS)  \geq \sigma_{\min}(\BZ^{\top}\BS) \geq n -\frac{\delta^2d\|\BDelta\|^2_{\op}}{2n}.
\]
\end{lemma}
This Lemma says that if $\BS$ is sufficiently close to $\BZ$, then $\BZ^{\top}\BS$ is  approximately an identity.
\begin{proof}
Note that 
\[
\|\BS - \BZ\BQ  \|_F^2 = 2nd - 2\lag \BQ, \BZ^{\top}\BS\rag
\]
where $\|\BS\|_F^2 = \|\BZ\BQ\|_F^2 = nd.$
Note that
\[
|\lag \BQ, \BZ^{\top}\BS \rag|\leq \|\BQ\|_{\op}\cdot \|\BZ^{\top} \BS\|_* = \|\BZ^{\top}\BS \|_*
\]
where $\|\BZ^{\top} \BS \|_*$ denotes the nuclear norm of $\BZ^{\top}\BS\in\RR^{d\times p}.$ The maximum is assumed if $\BQ = \BU\BV^{\top}$ where $\BU\in\RR^{d\times d}$ and $\BV\in\RR^{p\times d}$ are the left and right singular vectors of $\BZ^{\top}\BS.$

As a result, we get
\[
d^2_F(\BS,\BZ) = 2(nd - \|\BZ^{\top}\BS\|_*)\leq \frac{\delta^2d\|\BDelta\|^2_{\op}}{n}.
\]
Note that the largest singular value of $\BZ^{\top}\BS$ is at most $n$ which trivially follows from triangle inequality. For the smallest singular value of $\BZ^{\top}\BS$, we use the following inequality
\[
n - \sigma_{\min}(\BZ^{\top}\BS)\leq \sum_{i=1}^d (n - \sigma_i(\BZ^{\top}\BS)) = nd - \|\BZ^{\top}\BS \|_* \leq \frac{\delta^2d\|\BDelta\|^2_{\op}}{2n}
\]
which implies $\sigma_{\min}(\BZ^{\top}\BS) \geq n -2\delta^2n^{-1}d\|\BDelta\|^2_{\op}.$
\end{proof}

\begin{lemma}\label{lem:key}
Suppose a second-order critical point $\BS$ satisfies the proximity condition.
Then 
\begin{align*}
\lambda_{\min}(\BLambda_{ii}) & \geq n - \frac{\delta^2d\|\BDelta\|^2_{\op}}{2n} - \max_{1\leq i\leq n} \left\|  \sum_{j\neq i} \BDelta_{ij}\BS_j \right\|_{\op}, \\
\max_{1\leq i\leq n}\left\|\sum_{j\neq i} \BDelta_{ij}\BS_j \right\|_{\op} & \leq \delta\sqrt{\frac{d}{n}} \|\BDelta\|_{\op}\max_{1\leq i\leq n}\|\BDelta_i\|_{\op}  + \max_{1\leq i\leq n}\left\| \BDelta_{i}^{\top}\BZ\right\|_{\op}
\end{align*}
where $\BDelta_i$ is the $i$th block column of $\BDelta.$
\end{lemma}

\begin{proof}

Suppose $\BS$ is a SOCP with $d_F(\BS,\BZ)\leq \delta\sqrt{n^{-1}d}\|\BDelta\|_{\op}$. We have
\[
\sigma_{\min}(\BZ^{\top}\BS) \geq n - \frac{\delta^2 d\|\BDelta\|_{\op}^2}{2n}, \quad \BLambda_{ii}  =  \frac{1}{2}\sum_{j=1}^n (\BS_i\BS_j^{\top} \BA_{ji} + \BA_{ij}\BS_j\BS_i^{\top}) \succeq 0
\]
from Lemma~\ref{lem:sigmamin} and~\ref{lem:Lambdapos}. 
The first order necessary condition~\eqref{eq:grad} implies 
\[
\sum_{j=1}^n\BA_{ij}\BS_j= \BLambda_{ii}\BS_i, \qquad \BLambda_{ii} = \sum_{j=1}^n\BS_i\BS_j^{\top}\BA_{ji} = \sum_{j=1}^n\BA_{ij}\BS_j\BS_i^{\top}
\]
where $\BS_i\BS_i^{\top} = \I_d$. Therefore, the singular values of $\sum_{j=1}^n\BA_{ij}\BS_j$ and $\BLambda_{ii}$ are the same.
Moreover, due to the symmetry and $\BLambda_{ii}\succeq 0$, its eigenvalues and singular values match:
\begin{align*}
\lambda_{\min}(\BLambda_{ii}) & = \sigma_{\min}(\BLambda_{ii}) =  \sigma_{\min}\left( \sum_{j=1}^n \BA_{ij}\BS_j
\right) \\
& =  \sigma_{\min}\left(\sum_{j=1}^n \BS_j + \sum_{j=1}^n \BDelta_{ij}\BS_j \right) \\
& \geq \sigma_{\min}\left(\BZ^{\top}\BS \right) -  \left\|  \sum_{j\neq i} \BDelta_{ij}\BS_j\right\|_{\op} \\
& \geq n - \frac{\delta^2d\|\BDelta\|^2_{\op}}{2n} - \max_{1\leq i\leq n} \left\|  \sum_{j\neq i} \BDelta_{ij}\BS_j\right\|_{\op}
\end{align*}
where the lower bound is independent of $i$.
The key is to bound $\left\|  \sum_{j\neq i} \BDelta_{ij}\BS_j \right\|_{\op}$ which is suboptimal in this analysis. 
\begin{align*}
 \left\|\sum_{j\neq i}  \BDelta_{ij}\BS_j  \right\|_{\op} 
 & \leq  \left\|\sum_{j\neq i}  \BDelta_{ij}(\BS_j - \BQ) \right\|_{\op} +  \left\| \sum_{j\neq i}   \BDelta_{ij}\BQ\right\|_{\op} \\
 & \leq \|  \BDelta_{i}^{\top}(\BS - \BZ\BQ)  \|_{\op} + \| \BDelta_{i}^{\top}\BZ \|_{\op}
\end{align*}
where $\BDelta_i^{\top}\in\RR^{d\times nd}$ is the $i$th row block  of $\BDelta.$
The operator norm of $\| \BDelta_{i} (\BS - \BZ\BQ) \|_{\op} $ is bounded by
\[
 \|  \BDelta_{i}^{\top}(\BS - \BZ\BQ ) \|_{\op} 
\leq  \| \BDelta_i\|_{\op} \|\BS -\BZ\BQ\|_{\op}
\leq  \| \BDelta\|_{\op} \|\BS - \BZ\BQ \|_F.
\]
Thus we have
\begin{align*}
 \left\|\sum_{j\neq i}  \BDelta_{ij} \BS_j \right\|_{\op}  & \leq \| \BDelta_i\|_{\op} \|\BS - \BZ\BQ \|_F +  \left\|\BDelta_{i}^{\top} \BZ\right\|_{\op}  \leq \delta\sqrt{\frac{d}{n}} \|\BDelta\|_{\op}\|\BDelta_i\|_{\op}  + \left\|\BDelta_{i}^{\top} \BZ\right\|_{\op}.
 \end{align*}
Taking the maximum over $1\leq i\leq n$ gives the desired result.
\end{proof}

With this supporting lemma, we are ready to prove Proposition~\ref{prop:keys}.

\begin{proof}[\bf Proof of Proposition~\ref{prop:keys}]
The proof consists of two steps: first to show that $\BS$ is a global maximizer to~\eqref{def:BM} by showing that $\BLambda - \BA\succeq 0$; then prove that $\BS$ is exactly rank-$d$.

{\bf Step One: show that $\BS$ is a global maximizer}

Remember that $\BC\BS = 0$ if $\BS\in\RR^{nd\times p}$ is a critical point of $f.$ Thus, to show $\BC$ is positive semidefinite at critical point $\BS$, it suffices to test 
$\bu^{\top}\BC\bu \geq 0$
for all $\bu\in\RR^{nd\times 1}$ which is perpendicular to each column of $\BS$:
\[
\BS^{\top} \bu = 0 \in\RR^{p\times 1} \Longleftrightarrow \sum_{j=1}^n \BS_j^{\top} \bu_j = 0 \in\RR^{p}
\]
where $\bu_j\in\RR^d$ is the $j$th block of $\bu$, $1\leq j\leq n.$

Then it holds that
\begin{align*}
\bu^{\top}\BC \bu & = \bu^{\top}\BLambda\bu - \bu^{\top}\BA \bu \\
& \geq \lambda_{\min}(\BLambda) \|\bu\|^2 - \bu^{\top} (\BZ\BZ^{\top} + \BDelta )\bu \\
& \geq \lambda_{\min}(\BLambda) \|\bu\|^2 - \bu^{\top} \BZ\BZ^{\top} \bu -  \|\BDelta\|_{\op}\|\bu\|^2
\end{align*}
Note that $\lambda_{\min}(\BLambda) = \min_{1\leq i\leq n}\lambda_{\min}(\BLambda_{ii})$ which is given by Lemma~\ref{lem:key}.
For $\bu^{\top}\BZ \BZ^{\top} \bu$, we use $\sum_{j=1}^n \BS_j^{\top} \bu_j = 0$ and 
\begin{align*}
\bu^{\top}\BZ \BZ^{\top} \bu & = \left\| \sum_{j=1}^n \bu_j\right\|^2 =  \left\| \sum_{j=1}^n \BQ^{\top} \bu_j\right\|^2 \\
& = \left\|\sum_{j=1}^n (\BS_j - \BQ)^{\top} \bu_j \right\|^2 \\
& \leq  \|( \BS -\BZ \BQ)^{\top}\bu \|^2 \\
 & \leq \|\BS - \BZ\BQ  \|_{\op}^2 \|\bu\|^2 \\
 & \leq \frac{\delta^2d\|\BDelta\|_{\op}^2}{n}\|\bu\|^2
\end{align*}
where the last inequality uses the proximity condition. 
For $\lambda_{\min}(\BLambda)$, we apply Lemma~\ref{lem:key} and immediately arrive at:
\begin{align*}
\lambda_{\min}(\BC) & \geq n - \left(  \frac{3\delta^2d\|\BDelta\|^2_{\op}}{2n} + \delta\sqrt{\frac{d}{n}} \|\BDelta\|_{\op}\max_{1\leq i\leq n}\|\BDelta_i\|_{\op}  + \max_{1\leq i\leq n}\left\|\BDelta_{i}^{\top} \BZ\right\|_{\op} + \|\BDelta\|_{\op}
\right) \geq 0.
\end{align*}

\paragraph{Step Two: $\BS$ is exactly rank-$d$}

We have shown the solution to the Burer-Monteiro approach is equivalent to that of the SDP. Now, we will prove that the solution to the Burer-Monteiro approach is exactly rank $d$.

How to show that $\BC$ is rank-$d$ deficient? It suffices to bound the dimension of its null space.
In a more compact version, we have
\[
\BC = \BLambda - \BA = \BLambda - \BZ\BZ^{\top}  - \BDelta
\]
The null space is bounded by 
\begin{align*}
\Null(\BC) & = nd- \rank(\BLambda - \BZ \BZ^{\top}  - \BDelta) \\
& \leq nd + d - \rank(\BLambda - \BDelta)
\end{align*}
where
\[
\rank(\BLambda - \BZ \BZ^{\top}- \BDelta) + \rank(\BZ \BZ^{\top}) = \rank(\BLambda - \BZ \BZ^{\top} - \BDelta) + d  \geq \rank(\BLambda - \BDelta).
\]
It suffices to provide a lower bound of $\rank(\BLambda - \BDelta)$. In particular, we aim to show that $\BLambda - \BDelta$ is full-rank by 
\[
\BLambda - \BDelta\succ 0.
\]
This is guaranteed by
\[
\lambda_{\min}(\BLambda) > \|\BDelta\|_{\op}
\]
and more explicitly 
\begin{align*}
& \lambda_{\min}(\BLambda) - \|\BDelta\|_{\op} \\
& \quad \geq n - \left(\frac{\delta^2d\|\BDelta\|^2_{\op}}{2n} + \delta\sqrt{\frac{d}{n}} \|\BDelta\|_{\op}\max_{1\leq i\leq n}\|\BDelta_i\|_{\op}  + \max_{1\leq i\leq n}\left\|\BDelta_{i}^{\top} \BZ\right\|_{\op} + \|\BDelta\|_{\op} \right)  > 0.
\end{align*}
Then we have $\Null(\BC) \leq nd +d - \rank(\BLambda - \BDelta) = nd + d - nd = d.$
If $\BC$ is of rank $nd - d$, then $\rank(\BS)=d$. Thus the global optimum is the same as that of~\eqref{def:CVX} and~\eqref{def:OD}.
\end{proof}

\subsection{Proof of Proposition~\ref{prop:SZ}}\label{s:propSZ}

\begin{proof}[\bf Proof of Proposition~\ref{prop:SZ}]
It is unclear how to characterize the global maximizer to the objective function~\eqref{def:BM}. However, the global maximizer must be a 2nd critical point whose corresponding objective function value is greater than $f(\BS)$ evaluated at the fully synchronous state $\BS_i = \BS_j.$

Throughout our discussion, we let $\BQ$ be the minimizer to $\min_{\BQ\in \St(d,p)} \|  \BS - \BZ\BQ\|_F$.
Given $\BS$ which satisfies $f(\BS)\geq f(\BZ\BQ)$, we have
\[
f(\BS) \geq f(\BZ\BQ)  \Longleftrightarrow \lag \BZ\BZ^{\top} + \BDelta,\BS\BS^{\top}\rag \geq \lag \BZ\BZ^{\top} + \BDelta,\BZ\BZ^{\top}\rag.
\]
Note that $\lag \BZ\BZ^{\top},  \BZ\BZ^{\top}\rag = n^2 d$ and $\lag \BZ\BZ^{\top}, \BS\BS^{\top}\rag=\|\BZ^{\top}\BS\|_F^2.$ 
This gives
\begin{align}
n^2 d - \|\BZ^{\top}\BS\|_F^2
& \leq \lag \BDelta, \BS\BS^{\top}  -  \BZ\BZ^{\top} \rag \nonumber  \\
& =  \lag  \BDelta, ( \BS - \BZ\BQ)(\BS + \BZ\BQ)^{\top}\rag \nonumber  \\
& \leq  \| \BDelta(\BS + \BZ\BQ)\|_F \cdot d_F(\BS,\BZ) \label{eq:SZ_upper}
\end{align}
where $d_F(\BS,\BZ) = \|\BS - \BZ\BQ \|_F$ and
\[
( \BS - \BZ\BQ)(\BS + \BZ\BQ)^{\top} = \BS\BS^{\top} - (\BZ\BQ)^{\top}\BS + \BS(\BZ\BQ)^{\top} - \BZ\BZ^{\top}
\]

In fact, $n^2d - \|\BZ^{\top}\BS\|_F^2$ is well controlled by $d_F(\BS,\BZ):$
\begin{equation}\label{eq:SZ}
n^2 d -  \|\BZ^{\top}\BS\|_F^2 = \sum_{i=1}^d (n^2 - \sigma_i^2(\BZ^{\top}\BS) )  = \sum_{i=1}^d (n + \sigma_i(\BZ^{\top}\BS) )(n - \sigma_i(\BZ^{\top}\BS) )
\end{equation}
where  $\sigma_i(\BZ^{\top}\BS)$ is the $i$th largest singular value of $\BZ^{\top}\BS.$
On the other hand, it holds that
\[
\frac{1}{2}\min_{\BQ\in \St(d,p)}\|\BS - \BZ\BQ \|_F^2 = nd - \|\BZ^{\top}\BS\|_* = \sum_{i=1}^d (n - \sigma_i(\BZ^{\top}\BS))
\]
Remember that $0\leq \sigma_i(\BZ^{\top}\BS)\leq n$ due to the orthogonality of each $\BS_i$. Therefore, we have
\begin{equation}\label{eq:df2}
\frac{nd_F^2(\BS,\BZ)}{2} \leq  n (nd -  \|\BZ^{\top}\BS\|_*) \leq n^2 d -  \left\| \BZ^{\top} \BS \right\|^2_F  \leq 2n (nd - \|\BZ^{\top}\BS\|_*) \leq n d_F^2(\BS,\BZ)
\end{equation}
which follows from~\eqref{eq:SZ}.
Substitute it back into~\eqref{eq:SZ_upper}, and we get
\[
\frac{nd^2(\BS,\BZ)}{2} \leq n^2 d -  \left\| \BZ^{\top} \BS \right\|^2_F  \leq  \|\BDelta(\BS + \BZ\BQ)\|_F \cdot d(\BS,\BZ) 
\]
Immediately, we have the following estimate of $d(\BS,\BZ):$
\begin{align*}
d_F(\BS,\BZ) & \leq \frac{2}{n} \|\BDelta(\BS +\BZ \BQ) \|_F \\
& \leq \frac{2}{n} \cdot \|\BDelta\|_{\op} \|\BS + \BZ \BQ\|_F \\
& \leq \frac{2}{n} \cdot \|\BDelta\|_{\op} \cdot 2\sqrt{nd} \\
& \leq 4\sqrt{\frac{d}{n}} \|\BDelta\|_{\op} 
\end{align*}
where $\|\BS +\BZ \BQ\|_F\leq 2\sqrt{nd}$ follows from $\|\BS\|_F = \|\BZ\|_F = \sqrt{nd}.$
\end{proof}

\subsection{Proof of Proposition~\ref{prop:SZBM}}\label{s:propSZBM}
This section is devoted to proving all the SOCPs are highly aligned with the fully synchronized state. The proof follows from two steps: (a) using the second order necessary condition to show that all SOCPs have a large objective function value; (b) combining (a) with the first order necessary condition leads to Proposition~\ref{prop:SZBM}.

\begin{lemma}\label{lem:BM2nd}
All the second order critical points $\BS\in\St(d,p)^{\otimes n}$ must satisfy:
\begin{align*}
(p-d)\|\BZ^{\top}\BS\|_F^2  & \geq (p -2d)n^2 d + \|\BS\BS^{\top}\|_F^2d \\
&  + \sum_{i=1}^n\sum_{j=1}^n (\|\BS_i\BS_j^{\top}\|_F^2 - d)\Tr(\BDelta_{ij}) + (p-d)\lag \BDelta, \BZ\BZ^{\top} - \BS\BS^{\top}\rag.
\end{align*}
\end{lemma}
Suppose the noise is zero, then 
$(p-d)\|\BZ^{\top}\BS\|_F^2 \geq (p -2d)n^2 d + \|\BS\BS^{\top}\|_F^2d$ holds. It means that $\|\BZ^{\top}\BS\|_F^2$ is quite close to $n^2 d$, i.e., $\{\BS_i\}_{i=1}^n$ are highly aligned, if $p$ is reasonably large. The proof idea of this lemma can also be found in~\cite{MTG20,MMMO17}. 

\begin{proof}

Let's first consider the second-order necessary condition~\eqref{cond:2ndpsd}:
\[
\sum_{i=1}^n \lag \BLambda_{ii}, \dot{\BS}_i\dot{\BS}_i^{\top}\rag \geq \sum_{i=1}^n\sum_{j=1}^n\lag \BA_{ij}, \dot{\BS}_i\dot{\BS}_j^{\top}\rag
\]
for all $\dot{\BS}_i$ on the tangent space of $\St(p,d)$ at $\BS_i$ where
$\BLambda_{ii} = \frac{1}{2} \sum_{j=1}^n (\BS_i\BS_j^{\top}\BA_{ji} + \BA_{ij}\BS_j\BS_i^{\top} ).$
Now we pick $\dot{\BS}_i$ as
\[
\dot{\BS}_i = \BPhi(\I_p - \BS_i^{\top} \BS_i) \in\RR^{d\times p}
\]
where $\BPhi\in\RR^{d\times p}$ is a Gaussian random matrix, i.e., each entry in $\BPhi$ is an i.i.d. $\mathcal{N}(0,1)$ random variable. It is easy to verify that $\dot{\BS}_i$ is indeed on the tangent space since $\BS_i\dot{\BS}_i^{\top} = 0.$
By taking the expectation w.r.t. $\BPhi$, the inequality still holds:
\[
\sum_{i=1}^n \lag \BLambda_{ii}, \E \dot{\BS}_i\dot{\BS}_i^{\top}\rag \geq \sum_{i=1}^n\sum_{j=1}^n\lag \BA_{ij}, \E\dot{\BS}_i\dot{\BS}_j^{\top}\rag.
\]
It suffices to compute $\E\dot{\BS}_i\dot{\BS}_j^{\top} $ now.
\begin{align*}
\E (\dot{\BS}_i\dot{\BS}_j^{\top} ) & = \E \BPhi(\I_p - \BS_i^{\top}\BS_i)(\I_p - \BS_j^{\top}\BS_j)\BPhi^{\top} \\
& = \lag\I_p - \BS_i^{\top}\BS_i, \I_p - \BS_j^{\top}\BS_j \rag\I_d \\
& = (p - 2d + \|\BS_i\BS_j^{\top}\|_F^2)\I_d \\
& = 
\begin{cases}
(p- d)\I_d, & i = j, \\
(p-2d + \|\BS_i\BS_j^{\top}\|_F^2)\I_d, & i\neq j
\end{cases}
\end{align*}
where $d = \Tr(\BS_i\BS_i^{\top})$.
Therefore, we have
\begin{equation}\label{eq:2nd}
(p-d)\sum_{i=1}^n \Tr( \BLambda_{ii}) \geq \sum_{i=1}^n\sum_{j=1}^n(p-2d+\|\BS_i\BS_j^{\top}\|_F^2)\Tr(\BA_{ij}).
\end{equation}
The right hand side of~\eqref{eq:2nd} equals
\begin{align*}
& \sum_{i=1}^n\sum_{j=1}^n (p-2d + \|\BS_i\BS_j^{\top}\|_F^2)\Tr(\BA_{ij}) \\
& = \sum_{i=1}^n\sum_{j=1}^n (p-2d + \|\BS_i\BS_j^{\top}\|_F^2)(d + \Tr(\BDelta_{ij})) \\
& = (p -2d)n^2 d + \|\BS\BS^{\top}\|_F^2d + (p-2d)\lag \BDelta, \BZ\BZ^{\top}\rag + \sum_{i=1}^n\sum_{j=1}^n \|\BS_i\BS_j^{\top}\|_F^2 \Tr(\BDelta_{ij})
\end{align*}
where $\BA_{ij} = \I_d + \BDelta_{ij}.$
From the definition of $\BLambda_{ii}$, the left side of~\eqref{eq:2nd} equal to
\begin{align*}
\sum_{i=1}^n \Tr(\BLambda_{ii} ) & = \sum_{i=1}^n\sum_{j=1}^n \lag \BA_{ij}, \BS_i\BS_j^{\top}\rag\\
& = \lag \BA, \BS\BS^{\top}\rag \\
& = \lag \BZ\BZ^{\top} + \BDelta, \BS\BS^{\top}\rag\\
& = \|\BZ^{\top}\BS\|_F^2 + \lag \BDelta, \BS\BS^{\top}\rag.
\end{align*}
Plugging the estimation back to~\eqref{eq:2nd} results in
\begin{align*}
(p-d)\left( \|\BZ^{\top}\BS\|_F^2 + \lag \BDelta, \BS\BS^{\top}\rag \right) & \geq (p -2d)n^2 d + \|\BS\BS^{\top}\|_F^2d \\
& \quad + (p-2d)\lag \BDelta, \BZ\BZ^{\top}\rag + \sum_{i=1}^n\sum_{j=1}^n \|\BS_i\BS_j^{\top}\|_F^2 \Tr(\BDelta_{ij}).
\end{align*}
By separating the signal from the noise, we have
\begin{align*}
(p-d)\|\BZ^{\top}\BS\|_F^2  & \geq (p -2d)n^2 d + \|\BS\BS^{\top}\|_F^2d \\
& \quad + \sum_{i=1}^n\sum_{j=1}^n (\|\BS_i\BS_j^{\top}\|_F^2 - d)\Tr(\BDelta_{ij}) + (p-d)\lag \BDelta, \BZ\BZ^{\top}-  \BS\BS^{\top} \rag
\end{align*}
where $\lag \BDelta, \BZ\BZ^{\top}\rag = \sum_{i=1}^n\sum_{j=1}^n \Tr(\BDelta_{ij}).$
\end{proof}

\begin{lemma}\label{lem:BM1st}
Any first-order critical point satisfies:
\[
 \|\BS\BS^{\top}\|_F^2\geq \left\|\BZ^{\top}\BS \right\|_F^2  - \frac{1}{n}\|\BDelta\BS\|_F^2. 
\]
\end{lemma}
\begin{proof}
Note that the first-order necessary condition is
\[
\sum_{j=1}^n \BA_{ij}\BS_j - \frac{1}{2}\sum_{j=1}^n (\BS_i\BS_j^{\top}\BA_{ji} + \BA_{ij}\BS_j\BS_i^{\top})\BS_i = 0
\]
which implies $\sum_{j=1} \BS_i\BS_j^{\top}\BA_{ji}   = \sum_{j=1}\BA_{ij}\BS_j\BS_i^{\top}$ by applying $\BS_i^{\top}$ to the equation above.
Then it reduces to
\[
\sum_{j=1}^n \BA_{ij}\BS_j \left( \I_p - \BS_i^{\top} \BS_i\right) = 0 \Longleftrightarrow \left(\BZ^{\top}\BS + \sum_{j=1}^n \BDelta_{ij}\BS_j\right) \left( \I_p - \BS_i^{\top}\BS_i\right) = 0.
\]
By separating the signal from the noise, we have
\[
\BZ^{\top}\BS (\I_p - \BS_i^{\top} \BS_i) = -\sum_{j=1}^n \BDelta_{ij}\BS_j (\I_p - \BS_i^{\top} \BS_i)
\]
Taking the Frobenius norm leads to
\[
\|\BZ^{\top}\BS\|_F^2 - \lag \BZ^{\top}\BS \BS_i^{\top} \BS_i, \BZ^{\top}\BS\rag \leq \left\| \sum_{j=1}^n \BDelta_{ij}\BS_j \right\|_F^2
\]
where $\|\I_p - \BS_i^{\top} \BS_i\|_{\op}=1.$
Taking the sum over $1\leq i\leq n$ gives
\begin{align*}
\| \BDelta\BS\|_F^2 & =  \sum_{i=1}^n\left\| \sum_{j=1}^n \BDelta_{ij}\BS_j \right\|_F^2 \\
& \geq n \left\|\BZ^{\top}\BS \right\|_F^2 - \sum_{i=1}^n \lag \BZ^{\top}\BS \BS_i^{\top} \BS_i, \BZ^{\top}\BS \rag \\
& = n \left\|\BZ^{\top}\BS \right\|_F^2 - \lag \BZ^{\top}\BS \BS^{\top} \BS, \BZ^{\top}\BS \rag \\
& = n \left\|\BZ^{\top}\BS \right\|_F^2 - \lag \BS\BS^{\top}\BS\BS^{\top}, \BZ\BZ^{\top}, \rag
\end{align*}
where $\BS^{\top}\BS = \sum_{i=1}^n \BS_i^{\top}\BS_i.$
Thus
\[
\left\|\BZ^{\top}\BS \right\|_F^2  - \frac{1}{n}\|\BDelta\BS\|_F^2 \leq \frac{1}{n}\lag \BS\BS^{\top}\BS\BS^{\top}, \BZ\BZ^{\top} \rag \leq \|\BS\BS^{\top}\|_F^2
\]
where $\BZ\BZ^{\top} = \BJ_n\otimes \I_d \preceq n\I_{nd}.$
\end{proof}

\begin{lemma}\label{lem:hada}
For any matrix $\BX\in\RR^{nd\times nd}$, it holds that
\[
\|\BX\circ \BS\BS^{\top}\|_{\op} \leq \|\BX\|_{\op}
\]
where $\BS\in\RR^{p\times nd}$ and $[\BS\BS^{\top}]_{ii} = \I_d.$ Here $``\circ"$ stands for the Hadamard product of two matrices.
\end{lemma}
\begin{proof}
Let the $\bu_{\ell}\in\RR^{nd\times 1}$ be the $\ell$th column of $\BS$ where $1\leq \ell\leq p$ and let $\bvarphi\in\RR^{nd\times 1}$ be an arbitrary unit vector. 
\begin{align*}
\bvarphi^{\top}(\BX\circ \BS\BS^{\top})\bvarphi 
& = \sum_{\ell=1}^p \bvarphi^{\top}(\BX\circ \bu_{\ell}\bu_{\ell}^{\top})\bvarphi \\
& = \sum_{\ell=1}^p \bvarphi^{\top}\diag(\bu_{\ell})\BX \diag(\bu_{\ell})\bvarphi \\
& \leq \|\BX\|_{\op} \cdot \sum_{{\ell}=1}^p \|\diag(\bu_{\ell})\bvarphi\|^2 \\
& = \|\BX\|_{\op} \|\diag(\bvarphi) [\bu_1,\cdots,\bu_p] \|_F^2 \\
& = \|\BX\|_{\op} \cdot \| \diag(\bvarphi)\BS\|_F^2 \\
& = \|\BX\|_{\op} \cdot \Tr(\diag(\bvarphi)\BS\BS^{\top}\diag(\bvarphi)) \\
& =  \|\BX\|_{\op} \|\bvarphi\|^2 
\end{align*}
where $\diag(\BS\BS^{\top}) = \I_{nd}.$ Thus we have shown $\|\BX\circ \BS\BS^{\top}\|_{\op} \leq \|\BX\|_{\op}$.
\end{proof}

\begin{proof}[\bf Proof of Proposition~\ref{prop:SZBM}]
Lemma~\ref{lem:BM2nd} and~\ref{lem:BM1st} imply that all the SOCPs of~\eqref{def:BM} satisfy
\begin{align*}
(p-d)\|\BZ^{\top}\BS\|_F^2 
& \geq (p -2d)n^2 d + d\left\|\BZ^{\top}\BS \right\|_F^2  - \frac{d}{n}\|\BDelta\BS\|_F^2  \\
&  + \sum_{i=1}^n\sum_{j=1}^n (\|\BS_i\BS_j^{\top}\|_F^2 - d)\Tr(\BDelta_{ij}) + (p-d)\lag \BDelta, \BZ\BZ^{\top} - \BS\BS^{\top}\rag
\end{align*}
where  $\|\BS\BS^{\top}\|_F^2\geq \left\|\BZ^{\top}\BS \right\|_F^2  - \frac{1}{n}\|\BDelta\BS\|_F^2.$
This is equivalent to
\begin{align*}
(p-2d)(n^2 d - \|\BZ^{\top}\BS\|_F^2) & \leq  \underbrace{\frac{d}{n}\| \BDelta\BS\|_F^2}_{T_1} -  \underbrace{\sum_{i=1}^n\sum_{j=1}^n (\|\BS_i\BS_j^{\top}\|_F^2 - d) \Tr(\BDelta_{ij})}_{T_2} + (p-d)\underbrace{\lag \BDelta, \BS\BS^{\top} - \BZ\BZ^{\top} \rag}_{T_3} \\
& \leq |T_1| + |T_2| + (p-d)|T_3|.
\end{align*}

{\bf Estimation of $|T_1|$ and $|T_3|$:}
For $T_1$, we simply have
\[
|T_1| \leq \frac{d}{n}\|\BDelta\|_{\op}^2 \|\BS\|_F^2 =\frac{d}{n}\cdot \|\BDelta\|_{\op}^2\cdot nd = d^2\|\BDelta\|_{\op}^2.
\]

For $T_3$, we have
\begin{align*}
|T_3| & = |\lag \BDelta, \BS\BS^{\top} - \BZ\BZ^{\top}\rag| \\
& = |\lag \BDelta, (\BS - \BZ\BQ)(\BS + \BZ\BQ)^{\top}\rag| \\
& \leq \|\BDelta\|_{\op} \cdot \|\BS + \BZ\BQ\|_F \cdot \|\BS - \BZ\BQ\|_F \\
& \leq 2\|\BDelta\|_{\op}\sqrt{nd}\cdot d_F(\BS,\BZ).
\end{align*}

{\bf Estimation of $|T_2|$:}
Define a new matrix $\widetilde{\BDelta} : = \Tr_{d}(\BDelta)\otimes \BJ_d$ whose $(i,j)$-entry block is $\Tr(\BDelta_{ij})\BJ_d$ and $\|\widetilde{\BDelta}\|_{\op} = \|\Tr_d(\BDelta)\|_{\op}d$. Note that
\begin{align*}
\sum_{i=1}^n\sum_{j=1}^n (\|\BS_i\BS_j^{\top}\|^2_F- d)\Tr(\BDelta_{ij}) & = \sum_{i=1}^n\sum_{j=1}^n \lag \BS_i\BS_j^{\top}\circ \BS_i\BS_j^{\top} - \I_d, \BJ_d\rag \Tr(\BDelta_{ij}) \\
& = \sum_{i=1}^n\sum_{j=1}^n \lag \BS_i\BS_j^{\top}\circ \BS_i\BS_j^{\top} - \I_d, \widetilde{\BDelta}_{ij} \rag \\
& = \lag \BS\BS^{\top}\circ \BS\BS^{\top}- \BZ\BZ^{\top}\circ \BZ\BZ^{\top}, \widetilde{\BDelta}\rag \\
 & = \lag  (\BS\BS^{\top} - \BZ\BZ^{\top})\circ (\BS\BS^{\top} + \BZ\BZ^{\top}), \widetilde{\BDelta} \rag
\end{align*}
where $(\BS\BS^{\top}\circ \BS\BS^{\top})_{ij} = \BS_i\BS_j^{\top}\circ \BS_i\BS_j^{\top}$, $(\BZ\BZ^{\top}\circ \BZ\BZ^{\top})_{ij} =\I_d$, and $\|\BS_i\BS_j^{\top}\|_F^2 = \lag \BS_i\BS_j^{\top},  \BS_i\BS_j^{\top}\rag  = \lag \BS_i\BS_j^{\top}\circ \BS_i\BS_j^{\top}, \BJ_d\rag$.
As a result, we have
\begin{align*}
 \sum_{i=1}^n\sum_{j=1}^n (\|\BS_i\BS_j^{\top}\|^2_F - d) \Tr(\BDelta_{ij})  
  & = \lag  \BS\BS^{\top} - \BZ\BZ^{\top}, \widetilde{\BDelta} \circ (\BS\BS^{\top} + \BZ\BZ^{\top}) \rag  \\
& = \lag (\BS - \BZ\BQ)(\BS + \BZ\BQ)^{\top}, \widetilde{\BDelta} \circ (\BS\BS^{\top} + \BZ\BZ^{\top}) \rag \\
& \leq \|\BS - \BZ\BQ\|_F \cdot \| \widetilde{\BDelta} \circ (\BS\BS^{\top} + \BZ\BZ^{\top})\|_{\op} \|\BS + \BZ\BQ\|_F\\
& \leq d_F(\BS,\BZ)\cdot 2\sqrt{nd} \| \widetilde{\BDelta} \circ (\BS\BS^{\top} + \BZ\BZ^{\top})\|_{\op}. 
\end{align*}
Now the goal is to get an upper bound of $ \| \widetilde{\BDelta} \circ (\BS\BS^{\top} + \BZ\BZ^{\top})\|_{\op}$. In fact, it holds
\begin{align*}
\| \widetilde{\BDelta} \circ (\BS\BS^{\top} + \BZ\BZ^{\top})\|_{\op} 
& \leq \| \widetilde{\BDelta} \circ \BS\BS^{\top} \|_{\op}  + \| \widetilde{\BDelta} \circ  \BZ\BZ^{\top}\|_{\op} \\
& \leq 2\|\widetilde{\BDelta} \|_{\op} = 2d \|\Tr_{d}(\BDelta)\|_{\op}
\end{align*}
where the second inequality follows from Lemma~\ref{lem:hada} and $\|\widetilde{\BDelta}\|_{\op} = d\|\Tr_d(\BDelta)\|_{\op}$. Therefore,
\begin{align*}
|T_2| & \leq d_F(\BS,\BZ)\cdot 2\sqrt{nd} \| \widetilde{\BDelta} \circ (\BS\BS^{\top} + \BZ\BZ^{\top})\|_{\op} \\
& \leq d_F(\BS,\BZ) \cdot 2\sqrt{nd} \cdot 2d \|\Tr_d(\BDelta)\|_{\op} \\
& = 4d\sqrt{nd}  \|\Tr_d(\BDelta)\|_{\op}d_F(\BS,\BZ)
\end{align*}

\vskip0.25cm

Now we wrap up the calculations:
\begin{align*}
(p-2d)(n^2d - \|\BZ^{\top}\BS\|_F^2) & \leq d^2\|\BDelta\|_{\op}^2 +  4d\sqrt{nd}  \|\Tr_d(\BDelta)\|_{\op}d_F(\BS,\BZ) + 2(p-d)\|\BDelta\|_{\op}\sqrt{nd}\cdot d_F(\BS,\BZ) \\
& \leq d^2\|\BDelta\|_{\op}^2 + 2\sqrt{nd}(p-d+ 2\gamma d) \|\BDelta\|_{\op} \cdot d_F(\BS,\BZ)
\end{align*}
where $\gamma = \frac{\|\Tr(\BDelta)\|_{\op}}{\|\BDelta\|_{\op}} \vee 1$ is defined in~\eqref{def:gamma}.
Note that
$n^2 d - \|\BZ^{\top}\BS\|_F^2 \geq 2^{-1} nd_F^2(\BS,\BZ)$ in~\eqref{eq:df2}.
Thus for $p > 2d$, we have
\[
\frac{nd_F^2(\BS,\BZ)}{2} \leq \frac{2\sqrt{nd}(p-d+ 2\gamma d) \|\BDelta\|_{\op}}{p-2d} \cdot d_F(\BS,\BZ) + \frac{d^2}{p-2d}\|\BDelta\|_{\op}^2
\]
and equivalently
\[
d_F^2(\BS,\BZ) \leq \frac{4(p-d+ 2\gamma d) }{p-2d} \cdot \sqrt{\frac{d}{n}}\|\BDelta\|_{\op} \cdot d_F(\BS,\BZ) + \frac{2d}{p-2d}\cdot \frac{d}{n}\|\BDelta\|_{\op}^2.
\]
As a result, we have
\begin{align*}
d_F(\BS,\BZ) & \leq  \left( \frac{2(p-d+ 2\gamma d)}{p-2d}  +\sqrt{4\left(\frac{p-d+ 2\gamma d}{p-2d}\right)^2 + \frac{2d}{p-2d}}~\right)\sqrt{\frac{d}{n}}\|\BDelta\|_{\op}\\
& \leq  \frac{(2+\sqrt{5})(p-d+ 2\gamma d)}{p-2d} \sqrt{\frac{d}{n}}\|\BDelta\|_{\op} \leq \frac{(2+\sqrt{5})(p+ d)\gamma}{p-2d}  \sqrt{\frac{d}{n}}\|\BDelta\|_{\op}
\end{align*}
where $(p-d+2\gamma d)^2\geq (p+d)^2\geq 2d(p-2d)$ holds for $p \geq 2d + 1$ and $\gamma \geq 1.$
\end{proof}

\subsection{Proof of Theorem~\ref{thm:CVX_Gaussian} and~\ref{thm:BM_Gaussian}}
Proposition~\ref{prop:keys} implies that it suffices to prove
\begin{equation}\label{cond:suff}
n 
\geq  \frac{3\delta^2d}{2n}\|\BDelta\|^2_{\op} + \delta\sqrt{\frac{d}{n}} \|\BDelta\|_{\op}^2 + \max_{1\leq i\leq n}\left\|\BDelta_{i}^{\top} \BZ \right\|_{\op} + \|\BDelta\|_{\op}
\end{equation}
where $\max_{1\leq i\leq n}\|\BDelta_i\|_{\op} \leq \|\BDelta\|_{\op}.$
Now we will estimate
$\|\BDelta\|_{\op}$ and $\max_{1\leq i\leq n}\|\BDelta_i^{\top} \BZ\|_{\op}$ for $\BDelta = \sigma\BW$ where $\BW$ is an $nd\times nd$ symmetric Gaussian random matrix.

\begin{proof}[\bf Proof of Theorem~\ref{thm:CVX_Gaussian} and~\ref{thm:BM_Gaussian}]
The proof is straightforward: to show that~\eqref{cond:suff} holds for some $\delta$ in both convex and nonconvex cases. 
If $\BDelta = \sigma\BW$ where $\BW$ is a Gaussian random matrix, it holds that
\[
\|\BDelta\|_{\op} \leq 3\sigma\sqrt{nd}
\]
with high probability at least $1 - e^{-nd/2}$ according to~\cite[Proposition 3.3]{BBS17}. For $\BDelta_i^{\top} \BZ$, we have
\[
\BDelta_i^{\top}  \BZ= \sigma\sum_{j\neq i} \BW_{ij}\in\RR^{d\times d}, \quad (\BDelta_i^{\top}  \BZ)_{k\ell} \overset{\text{i.i.d.}}{\sim}\sigma\mathcal{N}(0, n-1), \quad 1\leq k,\ell\leq d.
\]
Theorem 4.4.5 in~\cite{V18} implies that the Gaussian matrix $\BDelta_i^{\top}\BZ$ is bounded by
\[
\| \BDelta_i^{\top}\BZ \|_{\op} \leq C_2\sigma\sqrt{n}(\sqrt{d} + \sqrt{2\log n})
\]
with probability at least $1 - 2n^{-2}$. By taking the union bound over all $1\leq i\leq n$, we have
\[
\max_{1\leq i\leq n}\| \BDelta_i^{\top} \BZ \|_{\op} \leq C_2\sigma\sqrt{n}(\sqrt{d} + \sqrt{2\log n})
\]
with probability at least $1 - 2n^{-1}$.
In the convex relaxation, we have $\delta = 4$. Then  the right hand of~\eqref{cond:suff} is bounded by
\[
C_1\left( \left( \frac{3\delta^2d}{2n} + \delta\sqrt{\frac{d}{n}}\right) 9\sigma^2 nd  + C_2\sigma\sqrt{n}(\sqrt{d} + \sqrt{2\log n}) + 3\sigma\sqrt{nd}\right)
\]
where $\delta = 4.$ The leading term is of order $\sigma^2 \sqrt{n}d^{3/2}$ and thus $\sigma < C_0 n^{1/4}d^{-3/4}$ guarantees the tightness of SDP.

For Burer-Monteiro approach, it suffices to estimate $\gamma$ in~\eqref{def:gamma}.
The partial trace  $\Tr_d(\BDelta)$ is essentially equal to $\sigma\sqrt{d} \BW_{GOE,n}$, which implies
\[
\|\Tr_d(\BDelta)\|_{\op} \leq 3\sigma\sqrt{nd}
\]
with probability at least $1 - e^{-n/2}$ and 
\[
\delta \|\BDelta\|_{\op} \leq  \frac{(2+\sqrt{5})(p+d)}{(p-2d)} \max\{ \|\BDelta\|_{\op}, \|\Tr_d(\BDelta)\|_{\op}\} \lesssim  \frac{p+d}{p-2d}\cdot \sigma \sqrt{nd}.
\]
where $\delta= (2+ \sqrt{5})(p+d)(p-2d)^{-1}\gamma.$

Thus the right hand of~\eqref{cond:suff} is bounded by
\[
C_1'\left(  \frac{d}{n}\cdot \left(\frac{p+d}{p-2d}\right)^2\cdot \sigma^2 nd + \sqrt{\frac{d}{n}}  \cdot  \frac{p+d}{p-2d}\cdot \sigma^2nd  + \sigma\sqrt{n}(\sqrt{d} + \sqrt{2\log n}) + \sigma\sqrt{nd}\right)
\]
for some universal constant $C_1'.$ The leading order term is $\sigma^2(p-2d)^{-1}(p+d)d\sqrt{nd} $
which implies that~\eqref{cond:suff} holds if
\[
\sigma^2 < \frac{C_0 n(p-2d)}{d\sqrt{nd}(p+d)} = \frac{C_0n^{1/2}(p-2d)}{d^{3/2}(p+d)}
\]
for some small constant $C_0$. This means $\sigma^2 < C_0 n^{1/2}(p-2d)d^{-3/2}(p+d)^{-1}$ ensures  that the optimization landscape of~\eqref{def:BM} is benign.
\end{proof}


\end{document}